\documentclass[11pt,a4paper]{article}

\usepackage[T1]{fontenc}
\usepackage[utf8]{inputenc}
\usepackage{lmodern}
\usepackage[a4paper,margin=1in]{geometry}
\usepackage{microtype}

\usepackage{graphicx}
\usepackage{dsfont}
\usepackage{bm}
\usepackage{amsmath,amsfonts,amssymb,amsthm,mathtools}

\usepackage[round,authoryear]{natbib}
\usepackage[hypertexnames=false,hidelinks]{hyperref}

\theoremstyle{plain}
\newtheorem{theorem}{Theorem}
\newtheorem{proposition}{Proposition}

\theoremstyle{definition}
\newtheorem{defn}{Definition}

\theoremstyle{remark}
\newtheorem{myrem}{Remark}

\DeclareMathOperator{\vecc}{vec}
\DeclareMathOperator{\diag}{diag}

\DeclareMathOperator{\Cov}{Cov}

\newcommand{\bs}[1]{\boldsymbol{#1}}
\newcommand{\ba}{\boldsymbol{\alpha}}

\newcommand{\one}{{\bm 1}}

\newcommand{\Es}{\mathbb{E}}
\newcommand{\EE}{\mathbb{E}}
\newcommand{\VV}{\mathbb{V}}

\newcommand{\cov}{\Cov}

\newcommand{\EM}{{\mathbf E}_{\rm M}}

\newcommand{\SM}{{\Sigma}_{\EM}}
\newcommand{\PM}{{\bf P}_{\!{\bf E}_{\rm M}{\bf E}_{\rm M}}}
\newcommand{\PMh}{\widehat{{\bf P}}_{\!{\bf E}_{\rm M}{\bf E}_{\rm M}}}
\newcommand{\PUM}{{\bf P}_{\!{\bf U}_{\rm M}{\bf U}_{\rm M}}}
\newcommand{\PUDM}{{\bf P}_{\!{\bf U}_{\rm M}{\bf D}_{\rm M}}}
\newcommand{\PDUM}{{\bf P}_{\!{\bf D}_{\rm M}{\bf U}_{\rm M}}}
\newcommand{\PDM}{{\bf P}_{\!{\bf D}_{\rm M}{\bf D}_{\rm M}}}

\newcommand{\WPM}{\mathbf W_{\!\PM}}

\newcommand{\Qm}{{\bf Q}_{\rm M}}
\newcommand{\Nm}{{\bf N}_{\rm M}}
\newcommand{\R}{{\bf R}}

\begin{document}

\title{Reliability inference for semi-Markov models based on multiple trajectories}

\author{%
Samis Trevezas$^{1,2,*}$ \qquad
Mohamed Hamdaoui$^{3}$ \qquad
Irene Votsi$^{4}$\\[0.8em]
\small
$^{1}$Department of Mathematics, National and Kapodistrian University of Athens,\\
\small Athens 15784, Greece\\
\small
$^{2}$MICS Laboratory, CentraleSup\'elec--Universit\'e Paris-Saclay,\\
\small Gif-sur-Yvette F-91190, France\\
\small
$^{3}$LEM3, Universit\'e de Lorraine, UMR CNRS 7239,\\
\small 7 Rue F\'elix Savart, Metz F-57000, France\\
\small
$^{4}$LIEC, Universit\'e de Lorraine, UMR CNRS 7360,\\
\small 8 rue Delestraint, Metz F-57000, France\\[0.6em]
\small
$^{*}$Corresponding author: \texttt{strevezas@math.uoa.gr}\\
\small
Mohamed Hamdaoui: \texttt{mohamed.hamdaoui@univ-lorraine.fr}\\
\small
Irene Votsi: \texttt{eirini.votsi@univ-lorraine.fr}
}

\date{}

\maketitle

\begin{abstract}
We develop nonparametric inference for reliability indicators of
discrete-time semi-Markov systems from independent trajectories observed
over a common fixed horizon. Augmenting the physical state by the backward
recurrence time yields a finite coupled Markov representation on the observed
age range. We distinguish the resulting age-restricted failure-or-exit time
from calendar truncation, since these two finite-horizon quantities coincide
only in special cases. The framework covers restricted factorial moments and
moment characteristics, calendar-truncated failure-time summaries, and the
discrete-time intensity of the hitting time. Under explicit row-exposure
conditions, strong consistency and joint asymptotic normality are established
for the empirical initial law, the required transition rows and the
corresponding plug-in functionals. The Gaussian random-matrix representation
gives pointwise and joint covariance formulas, simultaneous confidence
envelopes, Wald procedures for linear summaries, and curvature-adjusted
Gaussian approximations. Restriction diagnostics and a target-specific
horizon-selection rule based on exposure, boundary interaction and
nested-horizon stability are developed separately. Numerical experiments
assess the inferential formulas and the diagnostics, while a complete-case
illustration from the European Group for Blood and Marrow Transplantation
reports calendar-truncated failure-time summaries and finite-dimensional
hitting intensities.
\end{abstract}



\maketitle

\section{Introduction}
\label{S:1}

Semi-Markov models provide a natural framework for stochastic systems whose future evolution depends not only on the current state, but also on the time already spent in that state. Equivalently, the dynamics may be described through a Markov renewal chain and the associated process observed in calendar time. This representation allows non-geometric sojourn-time distributions and is therefore well suited to reliability, degradation modelling and related applications in which sojourn times carry essential information; see, for example, \citep{How71,Barb08,Chrys08,Barb16}.

In reliability applications, the state space is usually partitioned into operational and failed states. This leads to indicators such as availability, maintainability, failure rates, mean time to failure, variance of the time to failure and occurrence-rate-type quantities such as the discrete-time intensity of the hitting time. Nonparametric and plug-in estimation of such indicators has been studied in several semi-Markov settings; see, among others, \citep{Vas92,Sad02,Lim06,Barb16,Damic13,Damic15a,Mal14,Vots18,Vots15,Vots19,VB20}. Much of this literature is naturally formulated either for a fully specified semi-Markov law, for one long observed trajectory, or for reliability functionals that implicitly depend on the behaviour of the process over an unbounded future horizon.

\citet{Barb16} develop semi-Markov modelling for multi-state systems and the associated reliability characteristics, whereas \citet{VBouz25} study bootstrap procedures for semi-Markov reliability indicators. Their results concern, respectively, multi-state semi-Markov modelling and resampling-based reliability inference. The question treated here is which reliability functionals are identifiable when independent trajectories are observed up to a common fixed horizon, and how the corresponding plug-in estimators behave as the number of trajectories increases.

 We observe $L$ independent trajectories of the same semi-Markov system at the common times $0,1,\ldots,M$, where $M$ is fixed and $L$ tends to infinity. This multiple-trajectory structure is standard in survival and event-history analysis, where individual histories are used to estimate transition probabilities, occupation probabilities, survival probabilities and hitting-time distributions in multi-state models; see, for example, \citep{Putter2007,MeiraMachado2009}. Similar data structures arise in reliability studies, degradation monitoring, epidemiology and sensory analysis. For semi-Markov processes, the exact nonparametric likelihood theory of \citet{Trev11} provides an important basis for inference from one or several trajectories, while recent panel semi-Markov work, including \citet{Frascolla2022} and \citet{CardotFrascolla2024}, studies related multi-sample settings, mainly under parametric assumptions on the sojourn-time distributions. However, a nonparametric reliability theory for fixed-horizon multiple trajectories remains incomplete, especially for time-to-failure moments, moment characteristics, finite-dimensional intensity sequences, simultaneous confidence envelopes and horizon diagnostics.

The central difficulty is one of identifiability. By augmenting the semi-Markov state with the backward recurrence time, the process becomes a time-homogeneous Markov chain on an enlarged state space. Under fixed-horizon observation, only transition rows having positive exposure before time $M$ can be estimated, and only ages smaller than $M$ are represented. The unobserved tails of the sojourn-time distributions therefore remain unidentified. Consequently, unbounded-horizon quantities such as the full mean and variance of the time to failure are not nonparametric estimands unless additional tail assumptions are imposed.

This observation changes the target of inference. We first study the failure-or-exit time obtained by killing the coupled chain when it leaves the represented operational age set. Its moments, denoted by $\mathrm{MTTF}_M$ and $\mathrm{VTTF}_M$, are resolvent functionals of the finite coupled chain and are estimable under the stated positive-exposure condition for the transition rows on which the resolvent depends. We also distinguish the calendar-truncated variable $T_{\mathrm D}\wedge M$, which is bounded by the observation horizon and is a polynomial functional of the same finite restriction. The two variables agree in the alternating-renewal example considered below, but they need not agree when transitions between operational states reset the backward recurrence time. Neither quantity is presented as a universal substitute for its unbounded-horizon analogue.

Building on this finite-horizon perspective, the paper makes four contributions. First, we formulate the multiple-trajectory experiment as a finite Markov inference problem on the horizon-dependent coupled state space $E_M$, with inference restricted to transition rows having positive mean occupation before the horizon. We derive explicit expressions for the factorial moments, mean and variance of the age-restricted failure-or-exit time, and for the moments of the calendar-truncated failure time. This distinction is retained in the numerical applications.

Second, we develop joint nonparametric asymptotic inference for these restricted reliability indicators. With $M$ fixed and $L\to\infty$, we establish strong consistency and a joint Gaussian limit for the empirical restricted initial law and the required transition rows. Keeping the transition-matrix limit in Gaussian random-matrix form allows reliability estimators to be treated through finite-dimensional matrix asymptotic calculus, following the submitted manuscript of \citet{gavrilopoulos2026matrix}. This yields plug-in central limit theorems and covariance formulas for individual reliability functionals, including restricted factorial moments, restricted mean and variance of the time to failure, and related smooth moment characteristics.

Third, we derive finite-dimensional inference for the discrete-time intensity of the hitting time. For each $m\leq M$, the intensity is expressed as a matrix functional of the restricted coupled chain. Collecting these functionals over $m=1,\ldots,M$ gives a joint Gaussian limit for the whole finite intensity sequence. This directly supports pointwise confidence intervals, simultaneous Gaussian confidence envelopes and Wald-type procedures for linear summaries of the sequence. The same joint approach also gives inference for collections of restricted factorial moments and for derived shape summaries such as standard deviation, skewness and kurtosis.

Fourth, we introduce diagnostics that separate sampling uncertainty from finite-horizon restriction effects. A constructive non-identifiability result shows that the unbounded-horizon MTTF cannot be recovered from fixed-horizon trajectories without assumptions on the unobserved tail. Nested-horizon comparisons, maximum backward-recurrence times, near-boundary indicators and boundary-exit summaries are then used to assess interaction with the represented age boundary. When trajectories are available up to a common maximal horizon, these quantities lead to a target-specific analysis-horizon rule that combines row adequacy, a prespecified target-scale tolerance and paired nested-horizon uncertainty. These diagnostics do not estimate the unidentified continuation.

The paper is organized as follows. Section~\ref{S:2} introduces the semi-Markov framework, the coupled Markov representation and the fixed-horizon observation scheme. Section~\ref{S:restricted_indicators} defines and evaluates the restricted reliability indicators. Section~\ref{S:estimation} develops the empirical estimators, their asymptotic properties, the matrix-differential covariance formulas and the joint inference procedures. Section~\ref{S:diagnostics} presents the non-identifiability result and the diagnostics for finite-horizon restriction effects. Section~\ref{S:application} reports the simulation study, coverage assessments and the multiple-trajectory data example from the European Group for Blood and Marrow Transplantation (EBMT). Section~\ref{S:conclusion} concludes.

\section{Semi-Markov framework and notation}
\label{S:2}

This section fixes the semi-Markov framework used throughout the paper. We first recall the Markov renewal construction and the associated discrete-time semi-Markov chain. The reliability indicators considered later are functions not only of the current physical state, but also of the time elapsed since the last transition. For this reason, the backward recurrence time is introduced explicitly.

The augmentation by the backward recurrence time leads to a coupled process on $E\times\mathbb N_0$, which is Markov even when the physical state process alone is not. This representation is the basis for the finite-horizon restriction used in the sequel: under observations on $0,\ldots,M$, a transition row is estimable only when its state has positive mean occupation before time $M$. The notation below separates the full semi-Markov model from the finite-horizon coupled version used for inference.

\subsection{Semi-Markov kernel and associated process}

Throughout, $\mathbb{N}_0=\{0,1,2,\ldots\}$ and $\mathbb{N}=\{1,2,\ldots\}$. Vectors representing probability distributions are written as row vectors.
Let $E=\{1,\ldots,s\}$ be the finite state space of a discrete-time semi-Markov system. The Markov renewal construction is based on a state sequence $\mathbf J=(J_n)_{n\geq0}$, jump times $\mathbf S=(S_n)_{n\geq0}$ with $S_0=0$ and $0<S_1<S_2<\cdots$ almost surely, and sojourn times $X_n=S_n-S_{n-1}$ for $n\geq1$.

\begin{defn}
A matrix-valued sequence $q=(q_{ij}(k))$, $i,j\in E$, $k\in\mathbb{N}_0$, is a \emph{discrete-time semi-Markov kernel} if
\begin{equation*}
q_{ij}(k)\geq0,
\qquad
\sum_{j\in E}\sum_{k=0}^{\infty}q_{ij}(k)=1,
\qquad i\in E.
\end{equation*}
Throughout the paper we assume $q_{ij}(0)=0$ and $q_{ii}(k)=0$ for $k\geq1$.
\end{defn}

The homogeneous \emph{Markov renewal chain} is defined by
\begin{equation*}
\mathbb{P}\{J_{n+1}=j,S_{n+1}-S_n=k\mid J_{0:n},S_{0:n}\}
=q_{ij}(k),
\qquad J_n=i.
\end{equation*}
The associated \emph{semi-Markov chain} and the \emph{jump counting process} are given by
\begin{equation*}
Z_m=J_{N(m)},
\qquad
N(m)=\sup\{n\in\mathbb{N}_0:S_n\leq m\},
\qquad m\in\mathbb{N}_0.
\end{equation*}

For $i\in E$ and $k\geq0$, set
\begin{equation*}
\overline{H}_i(k)=\mathbb{P}(X_{n+1}>k\mid J_n=i).
\end{equation*}

The statistical results below concern the asymptotic setting in which $M$ is fixed and $L\to\infty$. The $L$ trajectories are assumed independent and identically distributed, completely observed at the times $0,1,\ldots,M$, and generated by a common homogeneous semi-Markov kernel and a common initial law. Every unknown transition row entering a stated functional is assumed to have positive mean occupation before the horizon. For the restricted moment functionals we also assume that the corresponding up-to-up matrix has spectral radius smaller than one. Irreducibility and positive recurrence are not required for the fixed-horizon independent-trajectory limit theorem.

\subsection{The coupled Markov chain}

The backward recurrence time is
\begin{equation*}
U_m=m-S_{N(m)},
\qquad m\in\mathbb{N}_0.
\end{equation*}
The process $\boldsymbol{(Z,U)}=((Z_m,U_m))_{m\geq0}$ is a time-homogeneous Markov chain; see, for instance, \citep{Chrys08}. Its initial distribution is denoted by $\ba$ and its transition matrix by ${\bf P}$. Since $S_0=0$, $\ba$ is supported by $\{(i,0):i\in E\}$.

For $u\geq0$ such that $\overline{H}_i(u)>0$, define
\begin{equation*}
p_{i,u;j}:=
\begin{cases}
\mathbb{P}(Z_{m+1}=j,U_{m+1}=0\mid Z_m=i,U_m=u), & j\neq i,\\[1mm]
\mathbb{P}(Z_{m+1}=i,U_{m+1}=u+1\mid Z_m=i,U_m=u), & j=i.
\end{cases}
\end{equation*}
Then
\begin{equation*}
p_{i,u;j}=
\begin{cases}
q_{ij}(u+1)/\overline{H}_i(u), & j\neq i,\\[1mm]
\overline{H}_i(u+1)/\overline{H}_i(u), & j=i.
\end{cases}
\end{equation*}
Consequently, for $k\geq1$ and $j\neq i$,
\begin{equation}
\label{eq:q_recovery_corrected}
q_{ij}(k)=\overline{H}_i(k-1)p_{i,k-1;j},
\end{equation}
whereas the survival function satisfies
\begin{equation*}
\overline{H}_i(k)=\overline{H}_i(k-1)p_{i,k-1;i},
\qquad
\overline{H}_i(0)=1.
\end{equation*}
Formula~\eqref{eq:q_recovery_corrected} is stated only for $j\neq i$, consistently with the convention $q_{ii}(k)=0$.

Let
\begin{equation*}
k_i=\sup\{k\geq1:q_{ij}(k)>0\textrm{ for some }j\in E\}.
\end{equation*}
The block corresponding to state $i$ is
\begin{equation*}
{\bf i}=\begin{cases}
\{(i,0),\ldots,(i,k_i-1)\}, & k_i<\infty,\\
\{i\}\times\mathbb{N}_0, & k_i=\infty.
\end{cases}
\end{equation*}
For a set $I\subset E$, write ${\bf I}=\cup_{i\in I}{\bf i}$. The coupled state space is ${\bf E}=\cup_{i\in E}{\bf i}$.

For the fixed observation horizon $M\geq1$, define the finite-horizon coupled state space
\begin{equation*}
\EM=\{(i,u):i\in E,\ u=0,1,\ldots,M-1\}\cap{\bf E}.
\end{equation*}
This is the finite part of the enlarged coupled state space induced by the observation horizon. The restriction of $\ba$ to $\EM$ is $\ba_{\EM}$ and the finite restricted transition matrix is
\begin{equation*}
\PM=\big(P(x,y)\big)_{x,y\in\EM}.
\end{equation*}
The row sum of $\PM$ is smaller than one only at a boundary row $(i,M-1)$ for which $\overline{H}_i(M)>0$. Thus $\PM$ is substochastic precisely when, for at least one state $i$ represented at age $M-1$, the sojourn-time support in state $i$ extends beyond $M$. In the statistical results this row affects inference only if it belongs to the fixed set of required rows and has positive mean occupation before the horizon.

More explicitly, if $x=(i,M-1)$, then a jump to a different physical state remains inside the finite restriction and satisfies
\begin{equation*}
P\bigl(x,(j,0)\bigr)=\frac{q_{ij}(M)}{\overline{H}_i(M-1)},
\qquad j\neq i,
\end{equation*}
whereas continuation of the same sojourn produces the transition to $(i,M)\notin\EM$ with probability
\begin{equation*}
p^{\mathrm{out}}_{x;M}
=
P\bigl(x,(i,M)\bigr)
=
\frac{\overline{H}_i(M)}{\overline{H}_i(M-1)}
=p_{i,M-1;i}.
\end{equation*}
Consequently,
\begin{equation*}
\sum_{y\in\EM}P(x,y)=1-p^{\mathrm{out}}_{x;M}.
\end{equation*}
Thus the missing row mass is the probability of entering the unrepresented continuation region. If $\bm p^{\mathrm{out}}_{M}$ collects these row deficits, the restriction may equivalently be completed by an absorbing exit state $\partial_M$ through the stochastic matrix
\begin{equation*}
\widetilde{\bf P}_{M}
=
\begin{pmatrix}
\PM & \bm p^{\mathrm{out}}_{M}\\
\bm 0^{\top} & 1
\end{pmatrix}.
\end{equation*}
The state $\partial_M$ records exit from $\EM$ and is absorbing. The block decomposition used below retains only transitions inside $\EM$; transitions to $\partial_M$ are represented by the substochastic row deficits.

\section{Evaluation of restricted reliability indicators}
\label{S:restricted_indicators}

Reliability indicators quantify the ability of a system to operate over time and provide summary measures for performance assessment, maintenance planning and risk evaluation. In semi-Markov models, such indicators combine the transition mechanism with the sojourn-time distributions and therefore account for non-geometric durations in operational and failed states. A functional of the coupled chain is nonparametrically identifiable from the fixed-horizon experiment only when it depends on the restricted initial law and on transition rows having positive mean occupation before the horizon. This section defines the age-restricted and calendar-truncated failure-time characteristics and the discrete-time intensity of the hitting time, and gives the finite matrix expressions used for their evaluation.

The state space $E$ is partitioned into operational states $\mathrm{U}$ and failure states $\mathrm{D}$. The corresponding subsets of the coupled state space are denoted by ${\bf U}$ and ${\bf D}$. At horizon $M$ we use
\begin{equation*}
{\bf U}_{\rm M}={\bf U}\cap\EM,
\qquad
{\bf D}_{\rm M}={\bf D}\cap\EM.
\end{equation*}
With the ordering $\EM={\bf U}_{\rm M}\cup{\bf D}_{\rm M}$, the finite matrix $\PM$ is written as
\begin{equation*}
\PM=
\begin{pmatrix}
\PUM & \PUDM\\
\PDUM & \PDM
\end{pmatrix}.
\end{equation*}
Here $\PUM=(\PM)_{{\bf U}_{\rm M}{\bf U}_{\rm M}}$, that is the restriction of $\PM$ to the rows and columns indexed by ${\bf U}_{\rm M}$. Similarly,
$\PUDM=(\PM)_{{\bf U}_{\rm M}{\bf D}_{\rm M}}$,
$\PDUM=(\PM)_{{\bf D}_{\rm M}{\bf U}_{\rm M}}$,
and
$\PDM=(\PM)_{{\bf D}_{\rm M}{\bf D}_{\rm M}}$
are obtained by retaining the rows and columns indexed by the corresponding subsets of $\EM$.

\subsection{Age-restricted failure-or-exit time}
Mean time to failure is a central reliability indicator because it summarizes, on the time scale of the system, the expected duration before the first entrance into a failure state. Nonparametric estimation of MTTF and related reliability indicators for semi-Markov processes was studied by \citet{Lim06}. In the discrete-time framework, \citet{VB20} developed confidence intervals for risk indicators, including the conditional MTTF, and illustrated the methodology with an application to wind energy production. We first define a failure-or-exit time on the represented operational age set. Inference for its moments requires positive mean occupation before time $M$ for every unknown transition row entering the corresponding resolvent.

The full time to failure of $Z$ is
\begin{equation*}
T_{\mathrm D}=\inf\{m\in\mathbb{N}_0:Z_m\notin\mathrm U\}.
\end{equation*}
Under fixed-horizon sampling, $T_{\mathrm D}$ is not identifiable in general. We therefore consider the stopping time obtained by leaving the restricted up set ${\bf U}_{\rm M}$. This is an age-restricted failure-or-exit time. It need not be bounded by the calendar horizon $M$, because a transition between distinct operational states resets the backward recurrence time.

\begin{defn}
\label{def:restricted_TTF}
The age-restricted failure-or-exit time is
\begin{equation*}
T_{\mathrm D,\rm M}=\inf\{m\in\mathbb{N}_0:(Z_m,U_m)\notin{\bf U}_{\rm M}\}.
\end{equation*}
The corresponding restricted mean and variance are
\begin{equation*}
\mathrm{MTTF}_{\rm M}=\Es(T_{\mathrm D,\rm M}),
\qquad
\mathrm{VTTF}_{\rm M}=\VV(T_{\mathrm D,\rm M}).
\end{equation*}
\end{defn}

The event $\{T_{\mathrm D,\rm M}>n\}$ means that the coupled chain remains in ${\bf U}_{\rm M}$ from time $0$ to time $n$. Define
\begin{equation*}
\Qm=\big(\PM\big)_{{\bf U}_{\rm M}{\bf U}_{\rm M}},
\qquad
\Nm=({\bf I}_{{\bf U}_{\rm M}}-\Qm)^{-1},
\end{equation*}
whenever the spectral radius $\rho(\Qm)<1$.

\begin{proposition}
\label{prop:restricted_pgf_moments}
Assume $\rho(\Qm)<1$. The probability generating function of $T_{\mathrm D,\rm M}$ is
\begin{equation}
\label{eq:pgf_restricted}
G_{\mathrm D,\rm M}(z)
=1+(z-1)\ba_{{\bf U}_{\rm M}}({\bf I}_{{\bf U}_{\rm M}}-z\Qm)^{-1}\one_{{\bf U}_{\rm M}},
\qquad |z|\leq1.
\end{equation}
For $k\geq1$, its $k$-th factorial moment is
\begin{equation}
\label{eq:MkTTF_M_population}
\mathrm{M_kTTF}_{\rm M}
=k!\,\ba_{{\bf U}_{\rm M}}\Qm^{k-1}\Nm^k\one_{{\bf U}_{\rm M}}.
\end{equation}
In particular,
\begin{equation}
\label{eq:MTTF_M_population}
\mathrm{MTTF}_{\rm M}
=\ba_{{\bf U}_{\rm M}}\Nm\one_{{\bf U}_{\rm M}},
\end{equation}
whereas
\begin{equation}
\label{eq:VTTF_M_population}
\mathrm{VTTF}_{\rm M}
=
\ba_{{\bf U}_{\rm M}}(2\Nm^2-\Nm)\one_{{\bf U}_{\rm M}}
-\big(\ba_{{\bf U}_{\rm M}}\Nm\one_{{\bf U}_{\rm M}}\big)^2.
\end{equation}
\end{proposition}

\begin{proof}
For $n\geq0$,
\begin{equation*}
\mathbb{P}(T_{\mathrm D,\rm M}>n)=\ba_{{\bf U}_{\rm M}}\Qm^n\one_{{\bf U}_{\rm M}}.
\end{equation*}
Using the standard identity for a non-negative integer-valued random variable,
\begin{equation*}
\mathbb{E}(z^T)=1+(z-1)\sum_{n=0}^{\infty}z^n\mathbb{P}(T>n),
\end{equation*}
we obtain \eqref{eq:pgf_restricted}. Since $\rho(\Qm)<1$, the inverse is finite-dimensional and analytic for $|z|\leq1$. Differentiating at $z=1$ gives \eqref{eq:MkTTF_M_population}. Formula~\eqref{eq:VTTF_M_population} follows from
\begin{equation*}
\VV(T)=\Es\{T(T-1)\}+\Es(T)-\Es(T)^2
\end{equation*}
and from $\Qm\Nm^2=\Nm^2-\Nm$.
\end{proof}

If every operational sojourn distribution that is reachable before failure has support contained in
\begin{equation*}
\{1,\ldots,M\},
\end{equation*}
then $T_{\mathrm D,\rm M}=T_{\mathrm D}$ almost surely. Otherwise, $T_{\mathrm D,\rm M}$ is the time to failure or exit from the represented operational age set.

\subsection{Calendar-truncated failure-time characteristics}
\label{subsec:calendar_truncation}

A second finite-horizon variable is obtained by truncating calendar time rather than the backward recurrence time.

\begin{defn}
\label{def:calendar_truncated_TTF}
The calendar-truncated time to failure is
\begin{equation*}
T_{\mathrm D}^{[M]}=T_{\mathrm D}\wedge M.
\end{equation*}
Its mean is denoted by
\begin{equation*}
\mathrm{MTTF}^{[M]}=\EE(T_{\mathrm D}^{[M]}).
\end{equation*}
\end{defn}

Unlike $T_{\mathrm D,\rm M}$, the variable $T_{\mathrm D}^{[M]}$ is bounded by $M$. The two variables coincide when, before first failure, no transition between distinct operational states can reset the backward recurrence time. They need not coincide in a model with several operational states. The following example shows explicitly how the two finite-horizon times differ.

\begin{myrem}
Take $E=\{1,2,3,4\}$, $\mathrm U=\{1,2,3\}$, $\mathrm D=\{4\}$ and $M=2$. Start from $(1,0)$ and let the semi-Markov kernel be deterministic, with
\begin{equation*}
q_{12}(2)=q_{23}(2)=q_{34}(2)=1.
\end{equation*}
Each transition between operational states resets the backward recurrence time. Hence the path remains in ${\bf U}_{\rm M}$ until the failure at calendar time $6$, so that
\begin{equation*}
T_{\mathrm D,\rm M}=T_{\mathrm D}=6,
\qquad
T_{\mathrm D}^{[2]}=2.
\end{equation*}
With $K$ successive operational states having deterministic sojourn length two, the corresponding difference is $2(K-1)$ and can therefore be arbitrarily large.

The two quantities answer different questions. The calendar-truncated variable is the natural target when $M$ is an administrative follow-up horizon and interest concerns failure within that calendar window. The age-restricted variable is the natural finite-state resolvent target when the restriction concerns the largest represented backward recurrence time in each operational state. For a non-repairable system, first-failure moments and the hitting-time mass function are primary summaries. For a repairable system, where several operational-to-failed transitions may occur, the DTIHT is preferable when the objective is the calendar-time occurrence probability of failures rather than the first-failure distribution alone.
\end{myrem}

\begin{proposition}
\label{prop:calendar_truncated_moments}
For $k\geq1$, let $(n)_0=1$ and $(n)_j=n(n-1)\cdots(n-j+1)$ for $j\geq1$. Then
\begin{equation}
\label{eq:calendar_truncated_factorial}
\EE\left\{(T_{\mathrm D}^{[M]})_k\right\}
=
k\sum_{n=0}^{M-1}(n)_{k-1}\,
\ba_{{\bf U}_{\rm M}}\Qm^n\one_{{\bf U}_{\rm M}}.
\end{equation}
In particular,
\begin{equation}
\label{eq:calendar_truncated_mean}
\mathrm{MTTF}^{[M]}
=
\ba_{{\bf U}_{\rm M}}
\left(\sum_{n=0}^{M-1}\Qm^n\right)
\one_{{\bf U}_{\rm M}}.
\end{equation}
\end{proposition}

\begin{proof}
For a non-negative integer-valued variable $T$,
\begin{equation*}
(T\wedge M)_k
=
k\sum_{n=0}^{M-1}(n)_{k-1}\mathds{1}_{\{T>n\}}.
\end{equation*}
For $n\leq M-1$, absence of failure up to time $n$ is equivalent to remaining in ${\bf U}_{\rm M}$ up to that time, and hence
\begin{equation*}
\mathbb{P}(T_{\mathrm D}>n)
=
\ba_{{\bf U}_{\rm M}}\Qm^n\one_{{\bf U}_{\rm M}}.
\end{equation*}
Taking expectations gives \eqref{eq:calendar_truncated_factorial}; the case $k=1$ gives \eqref{eq:calendar_truncated_mean}.
\end{proof}

\subsection{Discrete-time intensity of the hitting time}
\label{subsec:DTIHT_definition}

The discrete-time intensity of the hitting time (DTIHT) is a reliability indicator that describes the distribution, over discrete time, of entries into the failed set. It is used to quantify the temporal contribution of failures in systems with non-geometric sojourn times and can be evaluated through explicit semi-Markov or hidden renewal-type formulas. The indicator has been studied in several related settings. \citet{Vots14} presented an explicit formula for the DTIHT in hidden Markov renewal chains. \citet{Vots15} developed the evaluation and estimation of the DTIHT for semi-Markov chains and established asymptotic properties of a plug-in type estimator for hidden Markov renewal chains. Later on, \citet{Vots19} extended this line of work to hidden semi-Markov chains and introduced conditional counterparts of the DTIHT, together with joint inference for the corresponding plug-in estimators.

In contrast with the time to failure, which focuses on the first entrance into the failed set, the DTIHT is particularly relevant for repairable systems. In such systems, the process may return from failed states to operational states, so that several up-to-down transitions may occur along the same trajectory. The DTIHT then describes the occurrence rate of failures over calendar time, rather than only the distribution of a single terminal failure time.

For $m\geq1$, define
\begin{equation*}
r(m)=\mathbb{P}(Z_{m-1}\in\mathrm U,Z_m\in\mathrm D).
\end{equation*}
When the failed set is absorbing, $r(m)=\mathbb{P}(T_{\mathrm D}=m)$ and the DTIHT coincides with the probability mass function of the first hitting time. When repair transitions from $\mathrm D$ to $\mathrm U$ are allowed, the same expression should be interpreted as the probability of a failure occurrence, namely a transition from an operational to a failed state, at calendar time $m$. This convention is consistent with the failure-occurrence-rate viewpoint, in which failures are represented through counting processes. All inference below concerns this identifiable discrete-time intensity of the hitting time.

This quantity is exactly identifiable from the restriction to $\EM$ whenever $m\leq M$. Let
\begin{equation*}
{\bf U}_m={\bf U}\cap\{(i,u):u=0,\ldots,m-1\},
\qquad
{\bf D}_m={\bf D}\cap\{(i,u):u=0,\ldots,m-1\}.
\end{equation*}
With selector matrices $\R_{\bf A}$ extracting coordinates in ${\bf A}\subset\EM$, the finite expression is
\begin{equation*}
r(m;\rm M)
=\ba_{\EM}\PM^{m-1}\R_{{\bf U}_m}^{\top}\big(\PM\big)_{{\bf U}_m{\bf D}_m}\one_{{\bf D}_m},
\qquad 1\leq m\leq M.
\end{equation*}
Since the transition from an up state to a down state resets the backward recurrence time to zero, the block ${\bf D}_m$ may be replaced by the down states with backward recurrence time zero without changing the value.

\section{Estimation and inference}
\label{S:estimation}

This section develops the estimation of the restricted reliability indicators introduced above. It first defines the empirical initial law, its finite-horizon restriction, and the empirical restricted transition matrix, and then derives the asymptotic properties and joint inference procedures for the corresponding plug-in reliability estimators.

\subsection{Finite-horizon coupled chain estimation}
\label{subsec:est_restricted_transition_matrix}

Assume that $L$ independent trajectories are observed at times $0,1,\ldots,M$. For trajectory $\ell$, write
\begin{equation*}
X_m^{(\ell)}=(Z_m^{(\ell)},U_m^{(\ell)}),
\qquad m=0,\ldots,M.
\end{equation*}
The initial law of the coupled chain is the probability measure
\begin{equation*}
\ba_x=\mathbb{P}(X_0=x),
\qquad x\in{\bf E},
\end{equation*}
where $x$ denotes a coupled state, not only a physical state of $E$.
Its empirical version, based on the $L$ independent initial observations, is
\begin{equation*}
\widehat\ba_x(L)
=\frac{1}{L}\sum_{\ell=1}^{L}\mathds{1}_{\{X_0^{(\ell)}=x\}},
\qquad x\in{\bf E}.
\end{equation*}
Since the inference is carried out on the finite-horizon coupled state space $\EM$, we use the coordinate restriction of this empirical law to $\EM$:
\begin{equation*}
\widehat\ba_{\EM}(L)
=\big(\widehat\ba_x(L)\big)_{x\in\EM},
\qquad
\ba_{\EM}=\big(\ba_x\big)_{x\in\EM}.
\end{equation*}
Equivalently, for $x\in\EM$,
\begin{equation*}
\big(\widehat\ba_{\EM}(L)\big)_x
=\widehat\ba_x(L)
=\frac{N_x^{(0)}(L)}{L},
\qquad
N_x^{(0)}(L)=\sum_{\ell=1}^{L}\mathds{1}_{\{X_0^{(\ell)}=x\}}.
\end{equation*}
This restriction is not a conditional renormalization; it is simply the vector of the coordinates of the empirical initial law indexed by $\EM$. Under the present semi-Markov construction, $S_0=0$, so the initial distribution is supported by states of the form $(i,0)$, all of which belong to $\EM$ for every $M\geq1$. Hence $\widehat\ba_{\EM}(L)$ and $\ba_{\EM}$ have total mass one.

For $x,y\in\EM$, define the following aggregate counting processes over the $L$ observed trajectories. They record, respectively, the empirical occupation count of state $x$ before the horizon and the number of observed transitions from $x$ to $y$:
\begin{align*}
N_x(L)
&=\sum_{\ell=1}^{L}\sum_{m=0}^{M-1}\mathds{1}_{\{X_m^{(\ell)}=x\}},\\
N_{xy}(L)
&=\sum_{\ell=1}^{L}\sum_{m=0}^{M-1}\mathds{1}_{\{X_m^{(\ell)}=x,\ X_{m+1}^{(\ell)}=y\}}.
\end{align*}
The empirical restricted transition matrix is
\begin{equation*}
\PMh(L)=\big(\widehat P_{xy}(L)\big)_{x,y\in\EM},
\qquad
\widehat P_{xy}(L)=\frac{N_{xy}(L)}{N_x(L)},
\end{equation*}
with the convention $0/0=0$. Therefore the finite-horizon estimators used below are
\begin{equation}
\label{eq:MLE_restricted}
\big(\widehat\ba_{\EM}(L),\PMh(L)\big),
\end{equation}
which estimate the restricted initial law and restricted transition matrix $(\ba_{\EM},\PM)$.

For $x\in\EM$, define the mean occupation time up to time $M-1$ by
\begin{equation*}
\pi_{x\mid\rm M}
=\Es\left(\sum_{m=0}^{M-1}\mathds{1}_{\{X_m=x\}}\right)
=\sum_{m=0}^{M-1}(\ba_{\EM}\PM^m)(x).
\end{equation*}

Let
\begin{equation*}
B_x=\sum_{m=0}^{M-1}\mathds{1}_{\{X_m=x\}},
\qquad
\eta_{x\mid M}=\mathbb{P}(B_x>0).
\end{equation*}
Since $B_x$ is non-negative and bounded by $M$, one has $\pi_{x\mid\rm M}=0$ if and only if $B_x=0$ almost surely, equivalently $\eta_{x\mid M}=0$. Thus a state with zero mean exposure cannot be visited at any departure time $0,\ldots,M-1$. It may nevertheless occur for the first time at the terminal observation $X_M$; in that case it is an estimable destination coordinate, but its outgoing transition row is not observed. The central limit theorem therefore applies to any fixed deterministic collection of rows with positive exposure. A row with zero population exposure is omitted from the transition parameter and from the Gaussian limit. If a target depends on such a row, that target is not identifiable under the stated experiment.

For practical diagnostics, in addition to $N_x(L)$ and $\widehat\pi_{x\mid\rm M}(L)=N_x(L)/L$, it is useful to report
\begin{equation*}
D_x(L)=\sum_{\ell=1}^{L}\mathds{1}_{\{B_x^{(\ell)}>0\}},
\end{equation*}
the number of distinct trajectories exposing row $x$. The pair $(N_x(L),D_x(L))$ distinguishes a row supported by many trajectories from a row whose exposure is concentrated in only a few paths. Required rows with $N_x(L)=0$ must be reported explicitly; they cannot be regularized or deleted without changing the target.

The estimators in \eqref{eq:MLE_restricted} are the coupled-chain form of the exact nonparametric likelihood estimators of \citet{Trev11}, restricted to the states and transitions observable before the fixed horizon. The following theorem gives the strong consistency of the empirical restricted initial law and of the empirical restricted transition matrix, and establishes their joint asymptotic normality in the fixed-horizon, multiple-trajectory asymptotic setting. The Gaussian limit is stated in matrix form, which is the form needed below because the reliability estimators are obtained by applying finite-dimensional matrix differentials. A detailed proof of the asymptotic properties of the nonparametric likelihood estimators can be found in \citet{Trev11}. We also give a proof adapted to this framework for completeness.

The set of rows required by a reliability functional is fixed at the population level. For notational simplicity, Theorem~\ref{thm:CLT_aPM_restricted} is stated for the whole matrix $\PM$ under positive exposure of every row. The same argument applies to any fixed deterministic collection of rows with positive mean occupation.

\vspace{0.5em}
\begin{theorem}
\label{thm:CLT_aPM_restricted}
Assume that $\pi_{x\mid\rm M}>0$ for every $x\in\EM$. Then, as $L\to\infty$,
\begin{equation*}
\widehat{\ba}_{\EM}(L)\longrightarrow\ba_{\EM},
\qquad
\PMh(L)\longrightarrow\PM,
\qquad\textrm{almost surely}.
\end{equation*}
Moreover,
\begin{equation*}
\sqrt{L}\Big(\widehat{\ba}_{\EM}(L)-\ba_{\EM},\ \PMh(L)-\PM\Big)
\xrightarrow[]{d}
\big(\mathbf W_{\!\ba_{\EM}},\WPM\big),
\end{equation*}
where $\mathbf W_{\!\ba_{\EM}}$ and $\WPM$ are independent centered Gaussian objects satisfying
\begin{equation}
\label{eq:Var_Wa_restricted}
\VV\left(\mathbf W_{\!\ba_{\EM}}^{\top}\right)
=\diag(\ba_{\EM})-\ba_{\EM}^{\top}\ba_{\EM},
\end{equation}
and
\begin{equation}
\label{eq:Sigma_PM_restricted}
\SM=\VV\left(\vecc(\WPM^{\top})\right)
=\diag\left\{\frac{1}{\pi_{x\mid\rm M}}\Lambda_x,\ x\in\EM\right\}.
\end{equation}
Here, if ${\bf P}_x$ is row $x$ of $\PM$, then
\begin{equation*}
\Lambda_x=\diag({\bf P}_x)-{\bf P}_x^{\top}{\bf P}_x.
\end{equation*}
\end{theorem}

\begin{proof}
We give the argument because the row covariance in \eqref{eq:Sigma_PM_restricted} is the object on which all subsequent matrix differentials act. For one trajectory, put
\begin{equation*}
A_{xy}=\sum_{m=0}^{M-1}\mathds{1}_{\{X_m=x,X_{m+1}=y\}},
\qquad
B_x=\sum_{m=0}^{M-1}\mathds{1}_{\{X_m=x\}},
\end{equation*}
and let ${\bf A}_x=(A_{xy})_{y\in\EM}^{\top}$. Then
\begin{equation*}
\EE(B_x)=\pi_{x\mid\rm M},
\qquad
\EE({\bf A}_x)=\pi_{x\mid\rm M}{\bf P}_x^{\top}.
\end{equation*}
Since $M$ is fixed, these quantities form finite-dimensional random vectors for one observed trajectory; the corresponding sample counts are obtained by summing independent copies over the $L$ trajectories. The strong law applied to the $L$ independent copies gives
\begin{equation*}
L^{-1}N_x(L)\longrightarrow\pi_{x\mid\rm M},
\qquad
L^{-1}(N_{xy}(L))_{y\in\EM}\longrightarrow
\pi_{x\mid\rm M}{\bf P}_x^{\top}
\end{equation*}
almost surely. Hence, for every row with $\pi_{x\mid\rm M}>0$, the denominator is almost surely eventually positive and the ratio estimator is strongly consistent.

Let ${\bf e}_{X_{m+1}}^{\EM}$ denote the coordinate vector of $X_{m+1}$ when $X_{m+1}\in\EM$, and the null vector otherwise. For a fixed row $x$, define
\begin{equation*}
\bm S_x
={\bf A}_x-B_x{\bf P}_x^{\top}
=\sum_{m=0}^{M-1}\mathds{1}_{\{X_m=x\}}
\bigl({\bf e}_{X_{m+1}}^{\EM}-{\bf P}_x^{\top}\bigr).
\end{equation*}
The summands are martingale differences with respect to the natural filtration. Thus terms taken at different times are orthogonal. For two different rows $x\neq z$, the products at the same time vanish, and products at different times are again orthogonal by the martingale-difference property. Therefore
\begin{equation*}
\VV(\bm S_x)=\pi_{x\mid\rm M}\{\diag({\bf P}_x)-{\bf P}_x^{\top}{\bf P}_x\},
\qquad
\cov(\bm S_x,\bm S_z)=0,
\quad x\neq z.
\end{equation*}
The same formula remains valid when ${\bf P}_x$ is substochastic. This can occur only for a boundary state $x=(i,M-1)$ with positive continuation probability $\overline{H}_i(M)/\overline{H}_i(M-1)$; the transition outside $\EM$ is then an additional multinomial category which is not retained in the vector.

For the $L$ independent trajectories, the multivariate central limit theorem applied to the vector collecting $\mathds{1}_{\{X_0=x\}}$, $B_x$ and ${\bf A}_x$ for $x\in\EM$ gives a joint Gaussian limit. The ratio map is differentiable for every row $x$ such that
$\pi_{x\mid\rm M}>0$, that is, for every state with positive mean occupation
time before the horizon. Equivalently,
\begin{align*}
\widehat{\bf P}_x(L)^{\top}-{\bf P}_x^{\top}
&=\frac{L^{-1}\sum_{\ell=1}^L{\bf A}_x^{(\ell)}}{L^{-1}\sum_{\ell=1}^LB_x^{(\ell)}}
-{\bf P}_x^{\top} \\
&=\frac{1}{\pi_{x\mid\rm M}}\frac{1}{L}\sum_{\ell=1}^L
\bigl({\bf A}_x^{(\ell)}-B_x^{(\ell)}{\bf P}_x^{\top}\bigr)
+o_p(L^{-1/2}) \\
&=\frac{1}{\pi_{x\mid\rm M}}\frac{1}{L}\sum_{\ell=1}^L
\bm S_x^{(\ell)}+o_p(L^{-1/2}).
\end{align*}
Multiplying by $\sqrt L$ and using the covariance displayed above gives the row block $\pi_{x\mid\rm M}^{-1}\Lambda_x$ in \eqref{eq:Sigma_PM_restricted}; the off-diagonal row blocks are zero. The empirical restricted initial law is a multinomial average and has covariance \eqref{eq:Var_Wa_restricted}. Its covariance with the transition score is zero because
\begin{equation*}
\EE(\bm S_x\mid X_0)=0,
\qquad x\in\EM.
\end{equation*}
Since the joint limit is Gaussian, zero cross-covariance gives independence between $\mathbf W_{\!\ba_{\EM}}$ and $\WPM$.
\end{proof}

\subsection{Efficiency in the fixed-horizon likelihood}
\label{subsec:fixed_horizon_efficiency}
For fixed $M$, a fixed set of positively exposed rows and a fixed support pattern, the observed-data model is finite-dimensional. Each substochastic boundary row is completed by the exit category described above, so the parameter consists of the initial distribution and a product of multinomial row parameters. The estimators in \eqref{eq:MLE_restricted} are the maximum likelihood estimators in this model.

Assume that the true parameter lies in the relative interior of this fixed-support model. For one trajectory, retain the notation $B_x$ above and write ${\bf A}_x=(A_{xy})_{y\in\EM}^{\top}$. Let ${\bf H}(X_{0:M})$ be the matrix whose required row $x$ is
\begin{equation*}
{\bf H}_x(X_{0:M})
=
\pi_{x\mid\rm M}^{-1}\bigl({\bf A}_x-B_x{\bf P}_x^{\top}\bigr)^{\top},
\end{equation*}
and whose non-required rows are zero. For a differentiable scalar functional $\Phi(\ba_{\EM},\PM)$, define
\begin{equation*}
\operatorname{IF}_{\Phi}(X_{0:M})
=
D\Phi_{(\ba_{\EM},\PM)}
\left[
{\bf e}_{X_0}^{\top}-\ba_{\EM},
{\bf H}(X_{0:M})
\right].
\end{equation*}
This centered random variable is the first-order contribution of one complete trajectory. More precisely,
\begin{align*}
&\sqrt L\left\{
\Phi\bigl(\widehat\ba_{\EM}(L),\PMh(L)\bigr)
-
\Phi(\ba_{\EM},\PM)
\right\}
\\
&\qquad=
\frac{1}{\sqrt L}\sum_{\ell=1}^{L}
\operatorname{IF}_{\Phi}\bigl(X_{0:M}^{(\ell)}\bigr)
+o_p(1).
\end{align*}
Its variance is therefore the asymptotic variance of the plug-in estimator. The covariance in Theorem~\ref{thm:CLT_aPM_restricted} is the inverse Fisher-information covariance on the row-sum tangent spaces, including the complementary exit category at a substochastic boundary row. Consequently, the delta-method variance of the plug-in MLE is the information bound for every regular differentiable fixed-horizon functional. In this finite-dimensional sense, the proposed plug-in estimators are asymptotically efficient.

This efficiency statement concerns the regular finite-dimensional model induced by a fixed observation horizon. The paper does not consider semiparametric efficiency for the unrestricted semi-Markov kernel. Tail coordinates beyond the observation horizon are not identified by the fixed-horizon likelihood, whereas the targets studied here depend only on the identified finite restriction. The relevant comparison for the present estimators is therefore the Fisher-information bound in the fixed-horizon model. If zero transition probabilities are fixed in advance as structural zeros, the same argument applies on the corresponding lower-dimensional simplex, provided the remaining probabilities are positive. If the zero pattern is not fixed, or if the support is selected from the data, the parameter is on the boundary and the regular information-bound argument above does not apply. These cases are not considered here.

\subsection{Inference for individual reliability indicators}
\label{sec:asymp_indicators}

Following the matrix-differential approach of the submitted manuscript by \citet{gavrilopoulos2026matrix}, the proofs use the finite-dimensional delta method in matrix form. If ${\bf A}\subset\EM$, $\R_{\bf A}$ denotes the selector matrix such that $v_{\bf A}=\R_{\bf A}v$ and $B_{{\bf A}{\bf C}}=\R_{\bf A}B\R_{\bf C}^{\top}$. For row-wise vectorization,
\begin{equation*}
\vecc\{(B_{{\bf A}{\bf C}})^{\top}\}
=(\R_{\bf A}\otimes\R_{\bf C})\vecc(B^{\top}).
\end{equation*}
Set
\begin{equation*}
\Gamma_{\ba,\rm M}=\diag(\ba_{\EM})-\ba_{\EM}^{\top}\ba_{\EM},
\qquad
\SM=\VV\{\vecc(\WPM^{\top})\}.
\end{equation*}
For ${\bf A},{\bf C}\subset\EM$, write
\begin{equation*}
\Gamma_{\ba,\rm M}({\bf A})=\R_{\bf A}\Gamma_{\ba,\rm M}\R_{\bf A}^{\top},
\qquad
\SM({\bf A},{\bf C})=(\R_{\bf A}\otimes\R_{\bf C})\SM(\R_{\bf A}\otimes\R_{\bf C})^{\top}.
\end{equation*}

\subsubsection{Calendar-truncated mean}
\label{subsec:calendar_truncated_estimation}

The calendar-truncated mean in \eqref{eq:calendar_truncated_mean} has the matrix plug-in estimator
\begin{equation*}
\widehat{\mathrm{MTTF}}^{[M]}_{\rm mat}(L)
=
(\widehat\ba_{\EM}(L))_{{\bf U}_{\rm M}}
\left\{\sum_{n=0}^{M-1}\widehat{\bf Q}_{\rm M}(L)^n\right\}
\one_{{\bf U}_{\rm M}}.
\end{equation*}
It is a polynomial functional, so its first-order limit follows from Theorem~\ref{thm:CLT_aPM_restricted} by differentiating the finite sum of powers. The same target is estimated directly from the observed trajectories by
\begin{equation}
\label{eq:calendar_truncated_empirical}
\widehat{\mathrm{MTTF}}^{[M]}_{\rm emp}(L)
=
\frac{1}{L}\sum_{\ell=1}^{L}\left(T_{\mathrm D}^{(\ell)}\wedge M\right).
\end{equation}
Since the summands are independent and bounded,
\begin{equation*}
\sqrt{L}\left\{\widehat{\mathrm{MTTF}}^{[M]}_{\rm emp}(L)-\mathrm{MTTF}^{[M]}\right\}
\xrightarrow[]{d}
\mathcal{N}\left(0,\VV(T_{\mathrm D}^{[M]})\right).
\end{equation*}
Under the model and row-exposure conditions of Theorem~\ref{thm:CLT_aPM_restricted}, the matrix and direct empirical estimators have the same probability limit, but they are not generally identical in finite samples. The EBMT illustration uses \eqref{eq:calendar_truncated_empirical}.

\subsubsection{Restricted factorial moments}
\label{subsec:restricted_factorial_clt}

Let $\Qm=\big(\PM\big)_{{\bf U}_{\rm M}{\bf U}_{\rm M}}$ and $\Nm=({\bf I}_{{\bf U}_{\rm M}}-\Qm)^{-1}$. Define
\begin{equation*}
\widehat{\bf Q}_{\rm M}(L)=(\PMh(L))_{{\bf U}_{\rm M}{\bf U}_{\rm M}}.
\end{equation*}
On the event $\mathcal A_L=\{\rho(\widehat{\bf Q}_{\rm M}(L))<1\}$, set
\begin{equation*}
\widehat{\bf N}_{\rm M}(L)=({\bf I}_{{\bf U}_{\rm M}}-\widehat{\bf Q}_{\rm M}(L))^{-1}.
\end{equation*}
If $\rho(\Qm)<1$, consistency and continuity of the spectral radius imply that $\mathcal A_L$ occurs almost surely for all sufficiently large $L$. The plug-in estimators below may be assigned an arbitrary value on $\mathcal A_L^c$; this convention has no asymptotic effect. The plug-in estimator of \eqref{eq:MkTTF_M_population} is
\begin{equation*}
\widehat{\mathrm{M_kTTF}}_{\rm M}(L)
=k!\,(\widehat{\ba}_{\EM}(L))_{{\bf U}_{\rm M}}
\widehat{\bf Q}_{\rm M}(L)^{k-1}
\widehat{\bf N}_{\rm M}(L)^k
\one_{{\bf U}_{\rm M}}.
\end{equation*}

Define
\begin{equation*}
\bs\beta_{\rm M}^{(k)}=\Qm^{k-1}\Nm^k\one_{{\bf U}_{\rm M}}.
\end{equation*}
For a perturbation ${\bf H}$ of $\Qm$,
\begin{align*}
D\bs\beta_{\rm M}^{(k)}[{\bf H}]
&=
\sum_{j=0}^{k-2}\Qm^j{\bf H}\Qm^{k-2-j}\Nm^k\one_{{\bf U}_{\rm M}}
\\
&\quad+
\Qm^{k-1}\sum_{j=0}^{k-1}\Nm^{j+1}{\bf H}\Nm^{k-j}\one_{{\bf U}_{\rm M}},
\end{align*}
where the first sum is empty for $k=1$.

\begin{theorem}
\label{thm:CLT_MkTTF_restricted}
Assume the conditions of Theorem~\ref{thm:CLT_aPM_restricted} and $\rho(\Qm)<1$. Then, for every $k\geq1$,
\begin{equation*}
\sqrt{L}\left(\widehat{\mathrm{M_kTTF}}_{\rm M}(L)-\mathrm{M_kTTF}_{\rm M}\right)
\xrightarrow[]{d}
\mathcal{Z}^{(k)}_{\rm M},
\qquad
\mathcal{Z}^{(k)}_{\rm M}\sim\mathcal{N}(0,\sigma^2_{k,\rm M}),
\end{equation*}
where
\begin{equation}
\label{eq:Zk_operator_form_restricted}
\mathcal{Z}^{(k)}_{\rm M}
=k!\left\{(\mathbf W_{\!\ba_{\EM}})_{{\bf U}_{\rm M}}\bs\beta_{\rm M}^{(k)}
+\ba_{{\bf U}_{\rm M}}D\bs\beta_{\rm M}^{(k)}[\big(\WPM\big)_{{\bf U}_{\rm M}{\bf U}_{\rm M}}]\right\}.
\end{equation}
Moreover,
\begin{equation*}
\sigma^2_{k,\rm M}
=(k!)^2\left\{
(\bs\beta_{\rm M}^{(k)})^{\top}\Gamma_{\ba,\rm M}({\bf U}_{\rm M})\bs\beta_{\rm M}^{(k)}
+(\bs d_{\rm M}^{(k)})^{\top}\SM({\bf U}_{\rm M},{\bf U}_{\rm M})\bs d_{\rm M}^{(k)}
\right\},
\end{equation*}
with
\begin{align}
\label{eq:dk_explicit_restricted}
\bs d_{\rm M}^{(k)}
&=
\sum_{j=0}^{k-2}
\left\{(\Qm^j)^{\top}\ba_{{\bf U}_{\rm M}}^{\top}\right\}
\otimes
\left\{\Qm^{k-2-j}\Nm^k\one_{{\bf U}_{\rm M}}\right\}
\nonumber\\
&\quad+
\sum_{j=0}^{k-1}
\left\{(\Nm^{j+1})^{\top}(\Qm^{k-1})^{\top}\ba_{{\bf U}_{\rm M}}^{\top}\right\}
\otimes
\left\{\Nm^{k-j}\one_{{\bf U}_{\rm M}}\right\}.
\end{align}
\end{theorem}

\begin{proof}
The map $\Qm\mapsto\Qm^{k-1}({\bf I}-\Qm)^{-k}$ is differentiable on the open set $\{Q:\rho(Q)<1\}$. The derivative follows from the product rule, from
\begin{equation*}
D({\bf I}-Q)^{-1}[H]=({\bf I}-Q)^{-1}H({\bf I}-Q)^{-1},
\end{equation*}
and from the derivative of a matrix power. Applying Theorem~\ref{thm:CLT_aPM_restricted} and the delta method gives \eqref{eq:Zk_operator_form_restricted}. Independence of the two Gaussian components yields the variance decomposition, and 
\eqref{eq:dk_explicit_restricted} follows from $u^{\top}Hv=\vecc(H^{\top})^{\top}(u\otimes v)$.
\end{proof}

For $k=1$, this theorem gives the central limit theorem for $\widehat{\mathrm{MTTF}}_{\rm M}(L)$.

\subsubsection{Restricted VTTF}

Let
\begin{equation*}
\bm t_{\rm M}=\Nm\one_{{\bf U}_{\rm M}},
\qquad
\bm s_{\rm M}=\Nm^2\one_{{\bf U}_{\rm M}},
\qquad
m_{\rm M}=\ba_{{\bf U}_{\rm M}}\bm t_{\rm M}.
\end{equation*}
The restricted VTTF is
\begin{equation*}
\mathrm{VTTF}_{\rm M}
=\ba_{{\bf U}_{\rm M}}(2\Nm^2-\Nm)\one_{{\bf U}_{\rm M}}-m_{\rm M}^2.
\end{equation*}
Its derivative in the initial distribution direction is
\begin{equation*}
\bm c_{\rm M}
=(2\Nm^2-\Nm)\one_{{\bf U}_{\rm M}}-2m_{\rm M}\bm t_{\rm M}.
\end{equation*}
For a perturbation ${\bf H}$ of $\Qm$,
\begin{equation}
\label{eq:DVTTF_operator_restricted}
D_{\Qm}\mathrm{VTTF}_{\rm M}[{\bf H}]
=
\ba_{{\bf U}_{\rm M}}\{2\Nm^2-(1+2m_{\rm M})\Nm\}{\bf H}\bm t_{\rm M}
+2\ba_{{\bf U}_{\rm M}}\Nm{\bf H}\bm s_{\rm M}.
\end{equation}

\begin{theorem}
\label{thm:CLT_VTTF_restricted}
Assume the conditions of Theorem~\ref{thm:CLT_MkTTF_restricted}. Then
\begin{equation*}
\sqrt{L}\left(\widehat{\mathrm{VTTF}}_{\rm M}(L)-\mathrm{VTTF}_{\rm M}\right)
\xrightarrow[]{d}
\mathcal{Z}^{\rm V}_{\rm M},
\qquad
\mathcal{Z}^{\rm V}_{\rm M}\sim\mathcal{N}(0,\sigma^2_{\rm V,\rm M}),
\end{equation*}
where
\begin{equation*}
\mathcal{Z}^{\rm V}_{\rm M}
=(\mathbf W_{\!\ba_{\EM}})_{{\bf U}_{\rm M}}\bm c_{\rm M}
+D_{\Qm}\mathrm{VTTF}_{\rm M}[\big(\WPM\big)_{{\bf U}_{\rm M}{\bf U}_{\rm M}}].
\end{equation*}
Furthermore,
\begin{equation*}
\sigma^2_{\rm V,\rm M}
=\bm c_{\rm M}^{\top}\Gamma_{\ba,\rm M}({\bf U}_{\rm M})\bm c_{\rm M}
+(\bs d^{\rm V}_{\rm M})^{\top}\SM({\bf U}_{\rm M},{\bf U}_{\rm M})\bs d^{\rm V}_{\rm M},
\end{equation*}
with
\begin{equation*}
\bs d^{\rm V}_{\rm M}
=
\left[\{2\Nm^2-(1+2m_{\rm M})\Nm\}^{\top}\ba_{{\bf U}_{\rm M}}^{\top}\right]\otimes\bm t_{\rm M}
+
\left\{(2\Nm)^{\top}\ba_{{\bf U}_{\rm M}}^{\top}\right\}\otimes\bm s_{\rm M}.
\end{equation*}
\end{theorem}

\begin{proof}
The result is a direct application of the delta method to \eqref{eq:VTTF_M_population}. The derivative of $\Nm$ is $D\Nm[H]=\Nm H\Nm$ and the derivative of $\Nm^2$ is $\Nm H\Nm^2+\Nm^2H\Nm$. Rearranging gives \eqref{eq:DVTTF_operator_restricted}. The covariance formula follows by vectorizing the resulting linear form.
\end{proof}

\subsubsection{Discrete-time intensity of the hitting time at a fixed time point}

Fix $m\in\{1,\ldots,M\}$. Let
\begin{equation*}
\bm b_{m,\rm M}=\R_{{\bf U}_m}^{\top}\big(\PM\big)_{{\bf U}_m{\bf D}_m}\one_{{\bf D}_m},
\qquad
\bm v_{m,\rm M}=\PM^{m-1}\bm b_{m,\rm M}.
\end{equation*}
Then $r(m;\rm M)=\ba_{\EM}\bm v_{m,\rm M}$.
For a perturbation ${\bf H}$ of $\PM$,
\begin{align}
\label{eq:Dr_operator_restricted}
D_{\PM}r(m;\rm M)[{\bf H}]
&=
\ba_{\EM}\left(\sum_{j=0}^{m-2}\PM^j{\bf H}\PM^{m-2-j}\right)\bm b_{m,\rm M}
\nonumber\\
&\quad+
\ba_{\EM}\PM^{m-1}\R_{{\bf U}_m}^{\top}\R_{{\bf U}_m}{\bf H}\R_{{\bf D}_m}^{\top}\one_{{\bf D}_m},
\end{align}
where the sum is empty for $m=1$.

The two terms follow from the product rule. The matrix $\PM$ enters $r(m;\rm M)$ first through $\PM^{m-1}$, which determines the distribution at time $m-1$, and again through the final operational-to-failed block
\begin{equation*}
\bm b_{m,\rm M}(\PM)
=
\R_{{\bf U}_m}^{\top}\R_{{\bf U}_m}\PM\R_{{\bf D}_m}^{\top}\one_{{\bf D}_m}.
\end{equation*}
Therefore
\begin{equation*}
D\bm b_{m,\rm M}(\PM)[{\bf H}]
=
\R_{{\bf U}_m}^{\top}\R_{{\bf U}_m}{\bf H}\R_{{\bf D}_m}^{\top}\one_{{\bf D}_m},
\end{equation*}
which is the second term in \eqref{eq:Dr_operator_restricted}. When $m=1$, the map $\PM\mapsto\PM^0={\bf I}_{\EM}$ is constant and its derivative is zero. Hence the first sum is empty and only the derivative of the one-step operational-to-failed transition remains.

\medskip
\begin{theorem}
\label{thm:CLT_r_m_restricted}
Assume the conditions of Theorem~\ref{thm:CLT_aPM_restricted}. Then, for fixed $m\leq M$,
\begin{equation*}
\sqrt{L}\left(\widehat r(m;\rm M;L)-r(m;\rm M)\right)
\xrightarrow[]{d}
\mathcal{Z}^{r}_{m,\rm M},
\qquad
\mathcal{Z}^{r}_{m,\rm M}\sim\mathcal{N}(0,\sigma^2_r(m;\rm M)),
\end{equation*}
where
\begin{equation}
\label{eq:Zr_operator_form_restricted}
\mathcal{Z}^{r}_{m,\rm M}
=\mathbf W_{\!\ba_{\EM}}\bm v_{m,\rm M}
+D_{\PM}r(m;\rm M)[\WPM].
\end{equation}
Moreover,
\begin{equation*}
\sigma^2_r(m;\rm M)
=\bm v_{m,\rm M}^{\top}\Gamma_{\ba,\rm M}\bm v_{m,\rm M}
+(\bs d^r_{m,\rm M})^{\top}\SM\bs d^r_{m,\rm M},
\end{equation*}
with
\begin{align}
\label{eq:dr_explicit_restricted}
\bs d^r_{m,\rm M}
&=
\sum_{j=0}^{m-2}
\left\{(\PM^j)^{\top}\ba_{\EM}^{\top}\right\}
\otimes
\left\{\PM^{m-2-j}\bm b_{m,\rm M}\right\}
\nonumber\\
&\quad+
\left\{\R_{{\bf U}_m}^{\top}\R_{{\bf U}_m}(\PM^{m-1})^{\top}\ba_{\EM}^{\top}\right\}
\otimes
\left\{\R_{{\bf D}_m}^{\top}\one_{{\bf D}_m}\right\}.
\end{align}
\end{theorem}

\begin{proof}
The derivative of the power $\PM^{m-1}$ gives the first term in \eqref{eq:Dr_operator_restricted}; the derivative of the final operational-to-failed block gives the second term. The central limit theorem and the variance expression follow from Theorem~\ref{thm:CLT_aPM_restricted} and the row-wise vectorization identity $u^{\top}Hv=\vecc(H^{\top})^{\top}(u\otimes v)$.
\end{proof}

\subsection{Joint inference for restricted reliability indicators}
\label{subsec:joint_inference_restricted_indicators}

The individual asymptotic results obtained above all rely on the same empirical restricted initial law and the same empirical restricted transition matrix. They can therefore be combined to obtain joint inference for several reliability quantities evaluated on the common finite horizon. We first use this idea for the finite-dimensional sequence of the discrete-time intensity of the hitting time, where simultaneous confidence envelopes are often more informative than separate pointwise intervals. We then apply the same joint linear representation to restricted moment characteristics.

\subsubsection{Finite-dimensional DTIHT sequence}
\label{subsec:DTIHT_curve_joint}

The preceding theorem gives inference for the discrete-time intensity of the hitting time at a fixed time point. In applications, however, the object is often the whole finite sequence of the DTIHT over the observation horizon. For the fixed observation horizon $M$, set
\begin{equation*}
\bs r_{\rm M}=\bigl(r(1;\rm M),\ldots,r(M;\rm M)\bigr)^{\top},
\qquad
\widehat{\bs r}_{\rm M}(L)=\bigl(\widehat r(1;\rm M;L),\ldots,\widehat r(M;\rm M;L)\bigr)^{\top}.
\end{equation*}
Let
\begin{equation*}
{\cal V}_{r,\rm M}=\bigl(\bm v_{1,\rm M},\ldots,\bm v_{M,\rm M}\bigr),
\qquad
{\cal D}_{r,\rm M}=\bigl(\bs d^r_{1,\rm M},\ldots,\bs d^r_{M,\rm M}\bigr).
\end{equation*}
The matrix ${\cal V}_{r,\rm M}$ collects the derivatives with respect to the initial distribution, whereas ${\cal D}_{r,\rm M}$ collects the derivatives with respect to the restricted transition matrix.

\medskip
\begin{theorem}
\label{thm:CLT_r_curve_restricted}
Assume the conditions of Theorem~\ref{thm:CLT_aPM_restricted}. Then
\begin{equation*}
\sqrt{L}\left(\widehat{\bs r}_{\rm M}(L)-\bs r_{\rm M}\right)
\xrightarrow[]{d}
\boldsymbol{\mathcal{Z}}^{r}_{\rm M},
\qquad
\boldsymbol{\mathcal{Z}}^{r}_{\rm M}\sim \mathcal{N}_M(\bs 0,\Gamma^r_{\rm M}),
\end{equation*}
where
\begin{equation*}
\boldsymbol{\mathcal{Z}}^{r}_{\rm M}
=
{\cal V}_{r,\rm M}^{\top}\mathbf W_{\!\ba_{\EM}}^{\top}
+
{\cal D}_{r,\rm M}^{\top}\vecc(\WPM^{\top})
\end{equation*}
and
\begin{equation}
\label{eq:Gamma_r_curve_restricted}
\Gamma^r_{\rm M}
=
{\cal V}_{r,\rm M}^{\top}\Gamma_{\ba,\rm M}{\cal V}_{r,\rm M}
+
{\cal D}_{r,\rm M}^{\top}\SM {\cal D}_{r,\rm M}.
\end{equation}
\end{theorem}

\begin{proof}
For every fixed $m$, Theorem~\ref{thm:CLT_r_m_restricted} gives the first-order term. Since only finitely many times are involved, the Cram\'er--Wold device gives the joint convergence. The covariance matrix is obtained by assembling the same linear forms that appear in \eqref{eq:Zr_operator_form_restricted} and \eqref{eq:dr_explicit_restricted}.
\end{proof}

The above theorem illustrates the use of the matrix calculus. Replacing the unknown quantities in \eqref{eq:Gamma_r_curve_restricted} by their plug-in versions gives $\widehat\Gamma^r_{\rm M}(L)$. Pointwise intervals are
\begin{equation*}
\widehat r(m;\rm M;L)
\pm
z_{1-\gamma/2}\left\{\frac{\widehat\Gamma^r_{\rm M}(L)_{mm}}{L}\right\}^{1/2},
\qquad m=1,\ldots,M.
\end{equation*}
Here a pointwise confidence interval refers to one fixed ordinate $m$, whereas a simultaneous confidence envelope is a family of intervals constructed to cover the whole finite sequence $m=1,\ldots,M$ with probability $1-\gamma$ asymptotically.
For a simultaneous finite-dimensional confidence envelope, let $\widehat D_{r,\rm M}$ be the diagonal matrix with diagonal entries
\begin{equation*}
\{\widehat\Gamma^r_{\rm M}(L)_{mm}\}^{1/2},\qquad m=1,\ldots,M.
\end{equation*}
Coordinates with zero estimated variance are omitted. Let $c_{1-\gamma}$ be the plug-in Gaussian $(1-\gamma)$-quantile of
\begin{equation*}
\max_{1\leq m\leq M}|G_m|,
\qquad
\bs G\sim\mathcal{N}_M\left(\bs 0,
\widehat D_{r,\rm M}^{-1}\widehat\Gamma^r_{\rm M}(L)\widehat D_{r,\rm M}^{-1}\right),
\end{equation*}
then the intervals
\begin{equation*}
\widehat r(m;\rm M;L)
\pm
c_{1-\gamma}\left\{\frac{\widehat\Gamma^r_{\rm M}(L)_{mm}}{L}\right\}^{1/2},
\qquad m=1,\ldots,M,
\end{equation*}
have asymptotic simultaneous coverage $1-\gamma$. A Bonferroni confidence envelope is obtained by replacing $c_{1-\gamma}$ by $z_{1-\gamma/(2M)}$. The latter is simple but does not use the covariance structure of the curve.

The diagonal standardization makes the maximum statistic studentized: each ordinate is measured in its estimated standard-error units, while the plug-in correlation matrix retains the dependence between time points. Consistency of the plug-in covariance gives the stated asymptotic coverage, but no exact finite-sample coverage is implied. For moderate $L$, sparse required rows, highly correlated ordinates or very small estimated variances, the estimated correlation matrix and the resulting critical value can be unstable. The minimum row exposures and the omitted zero-variance coordinates should therefore be reported with the envelope.

A trajectory-level bootstrap provides an alternative. Resample the $L$ complete trajectories, recompute the transition estimator, the DTIHT sequence and its standard errors, and use the conditional quantile of
\begin{equation*}
\max_{m:\,\widehat\sigma_m>0}
\frac{\sqrt L\,|\widehat r^{\ast}(m;\rm M)-\widehat r(m;\rm M)|}
{\widehat\sigma_m^{\ast}}
\end{equation*}
as the simultaneous critical value. Resampling entire trajectories preserves their within-trajectory dependence.

A multiplier bootstrap can be constructed from the trajectory-level influence representation in Theorem~\ref{thm:CLT_r_curve_restricted}. Let $\widehat{\bm\psi}_{\ell}\in\mathbb R^M$ be the estimated influence vector of trajectory $\ell$ and let
\begin{equation*}
\overline{\bm\psi}=\frac{1}{L}\sum_{\ell=1}^{L}\widehat{\bm\psi}_{\ell}.
\end{equation*}
Generate independent multipliers $\xi_1,\ldots,\xi_L$ with mean zero and variance one, and set
\begin{equation*}
\bm Z^{\dagger}
=
\frac{1}{\sqrt L}\sum_{\ell=1}^{L}
\xi_{\ell}\bigl(\widehat{\bm\psi}_{\ell}-\overline{\bm\psi}\bigr).
\end{equation*}
The same multiplier is applied to all ordinates contributed by a trajectory, so the estimated dependence over time is retained. The conditional quantile of
\begin{equation*}
\max_{m:\,\widehat\sigma_m>0}
\frac{|Z_m^{\dagger}|}{\widehat\sigma_m}
\end{equation*}
is used as the simultaneous critical value. This procedure does not require re-estimation of the transition matrix for each multiplier draw. We do not establish a higher-order comparison between the two bootstrap procedures; see also \citet{VBouz25}.

Linear summaries of the curve are handled without new differentiations. If ${\bf C}$ is a fixed matrix, then
\begin{equation*}
\sqrt{L}\,{\bf C}\left(\widehat{\bs r}_{\rm M}(L)-\bs r_{\rm M}\right)
\xrightarrow[]{d}
\mathcal{N}\left(\bs 0,{\bf C}\Gamma^r_{\rm M}{\bf C}^{\top}\right).
\end{equation*}
For example, the cumulative discrete-time intensity up to time $m_0\leq M$ corresponds to ${\bf C}=(1,\ldots,1,0,\ldots,0)$. Similarly, linear restrictions ${\bf C}\bs r_{\rm M}=\bm c$ may be tested by the corresponding Wald statistic whenever ${\bf C}\Gamma^r_{\rm M}{\bf C}^{\top}$ has full row rank; otherwise a generalized inverse and the rank of this matrix are used.

\subsubsection{Restricted factorial moments and moment characteristics}
\label{subsec:joint_moment_characteristics}

The factorial-moment formulas are also naturally joint. Fix an integer $K\geq 1$ and define
\begin{align*}
\bs F_{K,\rm M}
&=\bigl(\mathrm{M_kTTF}_{\rm M}\bigr)_{k=1}^{K},\\
\widehat{\bs F}_{K,\rm M}(L)
&=\bigl(\widehat{\mathrm{M_kTTF}}_{\rm M}(L)\bigr)_{k=1}^{K}.
\end{align*}
Let
\begin{equation*}
\widetilde{\bm\beta}^{(k)}_{\rm M}=k!\,\bm\beta^{(k)}_{\rm M},
\qquad
\widetilde{\bs d}^{(k)}_{\rm M}=k!\,\bs d^{(k)}_{\rm M},
\qquad k=1,\ldots,K,
\end{equation*}
where $\bm\beta^{(k)}_{\rm M}$ and $\bs d^{(k)}_{\rm M}$ are those of Theorem~\ref{thm:CLT_MkTTF_restricted}. Put
\begin{equation*}
{\cal B}_{K,\rm M}=\bigl(\widetilde{\bm\beta}^{(1)}_{\rm M},\ldots,\widetilde{\bm\beta}^{(K)}_{\rm M}\bigr),
\qquad
{\cal D}_{K,\rm M}=\bigl(\widetilde{\bs d}^{(1)}_{\rm M},\ldots,\widetilde{\bs d}^{(K)}_{\rm M}\bigr).
\end{equation*}

\begin{theorem}
\label{thm:CLT_factorial_vector_restricted}
Assume the conditions of Theorem~\ref{thm:CLT_MkTTF_restricted}. Then
\begin{equation*}
\sqrt{L}\left(\widehat{\bs F}_{K,\rm M}(L)-\bs F_{K,\rm M}\right)
\xrightarrow[]{d}
\mathcal{N}_K(\bs 0,\Omega_{K,\rm M}),
\end{equation*}
where
\begin{equation*}
\Omega_{K,\rm M}
=
{\cal B}_{K,\rm M}^{\top}\Gamma_{\ba,\rm M}({\bf U}_{\rm M}){\cal B}_{K,\rm M}
+
{\cal D}_{K,\rm M}^{\top}\SM({\bf U}_{\rm M},{\bf U}_{\rm M}){\cal D}_{K,\rm M}.
\end{equation*}
\end{theorem}

\begin{proof}
The scalar limits in Theorem~\ref{thm:CLT_MkTTF_restricted} are linear functions of the same Gaussian pair
\begin{equation*}
\left((\mathbf W_{\!\ba_{\EM}})_{{\bf U}_{\rm M}},(\WPM)_{{\bf U}_{\rm M}{\bf U}_{\rm M}}\right).
\end{equation*}
Considering these linear forms jointly yields the stated multivariate Gaussian limit and the corresponding covariance matrix.
\end{proof}

For $K=4$, this result gives joint inference for the usual moment shape descriptors of $T_{D,\rm M}$. Denote the first four factorial moments by $F_1,F_2,F_3,F_4$ and the first four ordinary moments by $\eta_1,\eta_2,\eta_3,\eta_4$. Then
\begin{align*}
\eta_1&=F_1,\\
\eta_2&=F_2+F_1,\\
\eta_3&=F_3+3F_2+F_1,\\
\eta_4&=F_4+6F_3+7F_2+F_1.
\end{align*}
Consequently, on the set where the variance is positive, the map
\begin{equation*}
(F_1,F_2,F_3,F_4)
\longmapsto
\bs\theta_{\rm M}=(\mu_{\rm M},s_{\rm M},\gamma_{1,\rm M},\gamma_{2,\rm M})^{\top}
\end{equation*}
is smooth, where
\begin{align*}
\mu_{\rm M}&=\eta_1,\\
\nu_{\rm M}&=\eta_2-\eta_1^2,\\
s_{\rm M}&=\nu_{\rm M}^{1/2},\\
\gamma_{1,\rm M}&=\frac{\eta_3-3\eta_1\eta_2+2\eta_1^3}{\nu_{\rm M}^{3/2}},\\
\gamma_{2,\rm M}&=\frac{\eta_4-4\eta_1\eta_3+6\eta_1^2\eta_2-3\eta_1^4}{\nu_{\rm M}^{2}}.
\end{align*}
Here $\gamma_{1,\rm M}$ is the coefficient of skewness and $\gamma_{2,\rm M}$ is the Pearson coefficient of kurtosis, not excess kurtosis. If ${\bf J}_{h,\rm M}$ denotes the Jacobian of this map at $(F_1,F_2,F_3,F_4)$, then
\begin{equation*}
\sqrt{L}\left(\widehat{\bs\theta}_{\rm M}(L)-\bs\theta_{\rm M}\right)
\xrightarrow[]{d}
\mathcal{N}_4\left(\bs 0,
{\bf J}_{h,\rm M}\Omega_{4,\rm M}{\bf J}_{h,\rm M}^{\top}\right).
\end{equation*}
Writing ${\bf V}_{\theta,\rm M}={\bf J}_{h,\rm M}\Omega_{4,\rm M}{\bf J}_{h,\rm M}^{\top}$, this limit gives the usual Wald-type joint confidence region for the restricted mean, standard deviation, skewness and kurtosis. For a confidence level $1-\gamma$, its plug-in version is the ellipsoidal set
\begin{equation*}
\left\{\bs\theta:\, L\{\widehat{\bs\theta}_{\rm M}(L)-\bs\theta\}^{\top}\widehat{\bf V}_{\theta,\rm M}^{-1}\{\widehat{\bs\theta}_{\rm M}(L)-\bs\theta\}\leq \chi^2_{4,1-\gamma}\right\},
\end{equation*}
when $\widehat{\bf V}_{\theta,\rm M}$ is nonsingular. Marginal intervals are obtained from the diagonal entries. The same calculation applies to any smooth transformation of a finite number of restricted factorial moments.

\subsection{Finite-order refinements}
\label{subsec:higher_order_restricted}

The first-order limits above are sufficient for the consistency, asymptotic normality and confidence procedures developed in this paper. Some numerical comparisons below also use first- and second-order approximations, so we record here the corresponding finite-order stochastic expansions. Because the present functionals are finite-dimensional and smooth on the domain where the required inverses exist, repeated matrix differentiation gives the following expansions. Put
\begin{equation*}
\Delta_L=
\left({\bf H}_{\ba,L},{\bf H}_{{\bf P},L}\right)
=\left(
\sqrt{L}\{\widehat\ba_{\EM}(L)-\ba_{\EM}\},
\sqrt{L}\{\PMh(L)-\PM\}
\right).
\end{equation*}
If $\Phi(\ba,\PM)$ is twice continuously differentiable in a neighbourhood of the true value, then
\begin{align}
\label{eq:finite_order_expansion_restricted}
\Phi(\widehat\ba_{\EM}(L),\PMh(L))
&=\Phi(\ba_{\EM},\PM)
  +L^{-1/2}D\Phi_{(\ba_{\EM},\PM)}[\Delta_L]
\nonumber\\
&\quad+
\frac{1}{2L}
D^2\Phi_{(\ba_{\EM},\PM)}[\Delta_L,\Delta_L]
+o_p(L^{-1}).
\end{align}
This formula is not used to prove consistency or the first-order variances. It gives the second-order term in an approximation on the original scale. A complete order $L^{-1}$ bias correction would also require the order $L^{-1}$ bias of the estimators of $\ba_{\EM}$ and $\PM$.

For the restricted mean, define
\begin{equation*}
\Phi_{\rm M}(\ba,\PM)=\ba_{{\bf U}_{\rm M}}\Nm\one_{{\bf U}_{\rm M}},
\qquad
\bm t_{\rm M}=\Nm\one_{{\bf U}_{\rm M}}.
\end{equation*}
In the direction $(\bm h,{\bf H})$, the first differential is
\begin{equation}
\label{eq:DMTTF_restricted}
D\Phi_{\rm M}[\bm h,{\bf H}]
=
\bm h_{{\bf U}_{\rm M}}\bm t_{\rm M}
+
\ba_{{\bf U}_{\rm M}}\Nm
{\bf H}_{{\bf U}_{\rm M}{\bf U}_{\rm M}}
\bm t_{\rm M}.
\end{equation}
The second-order term, that is, one half of the second differential evaluated twice in the same direction, is
\begin{align*}
C^{\rm MTTF}_{\rm M}[\bm h,{\bf H}]
&=
\bm h_{{\bf U}_{\rm M}}\Nm
{\bf H}_{{\bf U}_{\rm M}{\bf U}_{\rm M}}
\bm t_{\rm M}
\\
&\quad+
\ba_{{\bf U}_{\rm M}}\Nm
{\bf H}_{{\bf U}_{\rm M}{\bf U}_{\rm M}}
\Nm
{\bf H}_{{\bf U}_{\rm M}{\bf U}_{\rm M}}
\bm t_{\rm M}.
\end{align*}
Hence the Gaussian approximation including the quadratic curvature term is obtained from
\begin{equation}
\label{eq:MTTF_first_second_approx_restricted}
\mathrm{MTTF}_{\rm M}
+L^{-1/2}D\Phi_{\rm M}[\mathbf W_{\!\ba_{\EM}},\WPM]
+L^{-1}C^{\rm MTTF}_{\rm M}[\mathbf W_{\!\ba_{\EM}},\WPM].
\end{equation}
The first-order Gaussian approximation is obtained from \eqref{eq:MTTF_first_second_approx_restricted} by deleting the last term. The quadratic approximation is a curvature-adjusted Gaussian approximation; no Edgeworth-order accuracy is claimed.

For the intensity ordinate, put
\begin{equation*}
{\bf J}_{m,\rm M}=\R_{{\bf U}_m}^{\top}\R_{{\bf U}_m},
\qquad
\bm e_{m,\rm M}=\R_{{\bf D}_m}^{\top}\one_{{\bf D}_m},
\end{equation*}
so that
\begin{equation*}
r(m;\rm M)=\ba_{\EM}\PM^{m-1}{\bf J}_{m,\rm M}\PM\bm e_{m,\rm M}.
\end{equation*}
The corresponding second-order term is
\begin{align*}
C^r_{m,\rm M}[\bm h,{\bf H}]
&=
\bm h
\left\{
\sum_{j=0}^{m-2}
\PM^j{\bf H}\PM^{m-2-j}{\bf J}_{m,\rm M}\PM
+
\PM^{m-1}{\bf J}_{m,\rm M}{\bf H}
\right\}
\bm e_{m,\rm M}
\\
&\quad+
\ba_{\EM}
\left\{
\sum_{a+b+c=m-3}
\PM^a{\bf H}\PM^b{\bf H}\PM^c{\bf J}_{m,\rm M}\PM
\right.
\\
&\hspace{28mm}\left.
+
\sum_{j=0}^{m-2}
\PM^j{\bf H}\PM^{m-2-j}{\bf J}_{m,\rm M}{\bf H}
\right\}
\bm e_{m,\rm M},
\end{align*}
with the convention that empty sums are zero. In numerical work one may simulate the limiting Gaussian pair and compare the first-order and second-order approximations of the plug-in estimator.

\begin{myrem}
\label{rem:second_order_bias_MLE}
The second-order term in \eqref{eq:finite_order_expansion_restricted} is not, by itself, a bias correction. The transition estimator is a ratio estimator. For one trajectory define
\begin{equation*}
A_{xy}=\sum_{m=0}^{M-1}\mathds{1}_{\{X_m=x,X_{m+1}=y\}},
\qquad
B_x=\sum_{m=0}^{M-1}\mathds{1}_{\{X_m=x\}}.
\end{equation*}
If $\pi_{x\mid\rm M}>0$, then the probability that the total occupation count of row $x$ is zero is exponentially small, and the second-order expansion of the ratio gives
\begin{equation}
\label{eq:transition_ratio_bias_restricted}
\EE\{\widehat P_{xy}(L)\}-P(x,y)
=
\frac{1}{L}
\frac{P(x,y)\VV(B_x)-\cov(A_{xy},B_x)}{\pi_{x\mid\rm M}^{2}}
+o(L^{-1}).
\end{equation}
Denote the coefficient of $L^{-1}$ in \eqref{eq:transition_ratio_bias_restricted} by $B_{P,xy}$. Since the empirical restricted initial law is unbiased, there is no analogous term for $\ba_{\EM}$. If, in addition, the second-order Taylor remainder is uniformly integrable, then for a twice continuously differentiable scalar functional $\Phi(\ba,\PM)$,
\begin{align*}
&\EE\{\Phi(\widehat\ba_{\EM}(L),\PMh(L))\}
-
\Phi(\ba_{\EM},\PM)
\\
&\quad=
\frac{1}{L}
\left[
D_{\bf P}\Phi[B_P]
+
\frac{1}{2}
\EE\left\{
D^2\Phi[(\mathbf W_{\!\ba_{\EM}},\WPM),(\mathbf W_{\!\ba_{\EM}},\WPM)]
\right\}
\right]
+o(L^{-1}).
\end{align*}
A full analytical correction therefore requires two finite-horizon covariance terms: $\VV(B_x)$ and $\operatorname{Cov}(A_{xy},B_x)$. They can be computed by finite-state recursions or estimated by a parametric bootstrap from the fitted restricted coupled chain. In the numerical section below, the second-order term is used only to form a second-order approximation.

These quantities are also directly estimable from the independent trajectories because $A_{xy}^{(\ell)}$ and $B_x^{(\ell)}$ are observed trajectory-level counts. Writing $\overline{A}_{xy}=L^{-1}\sum_{\ell}A_{xy}^{(\ell)}$ and $\overline{B}_x=L^{-1}\sum_{\ell}B_x^{(\ell)}$, one may use
\begin{align*}
\widehat{\VV}(B_x)
&=
\frac{1}{L-1}\sum_{\ell=1}^{L}
\bigl(B_x^{(\ell)}-\overline{B}_x\bigr)^2,\\
\widehat{\operatorname{Cov}}(A_{xy},B_x)
&=
\frac{1}{L-1}\sum_{\ell=1}^{L}
\bigl(A_{xy}^{(\ell)}-\overline{A}_{xy}\bigr)
\bigl(B_x^{(\ell)}-\overline{B}_x\bigr).
\end{align*}
Substitution in \eqref{eq:transition_ratio_bias_restricted} estimates the ratio-bias coefficient $B_{P,xy}$. Exact model-based values can instead be obtained by finite-horizon forward recursions for the second moments of the occupation and transition counts, while a trajectory bootstrap can estimate their combined finite-sample effect without separating the two contributions. The numerical comparison in this paper uses only the curvature term and is not labelled as a bias-corrected estimator.

\end{myrem}

The variances and covariance matrices in Theorems~\ref{thm:CLT_MkTTF_restricted},~\ref{thm:CLT_VTTF_restricted},~\ref{thm:CLT_r_m_restricted},~\ref{thm:CLT_r_curve_restricted} and~\ref{thm:CLT_factorial_vector_restricted} are estimated by replacing $\ba_{\EM}$ and $\PM$ with $\widehat{\ba}_{\EM}(L)$ and $\PMh(L)$. The mean occupation time $\pi_{x\mid\rm M}$ is estimated by its empirical counterpart $\widehat\pi_{x\mid\rm M}(L)=N_x(L)/L$ in \eqref{eq:Sigma_PM_restricted}. The row indices required by a target are fixed in advance. If $N_x(L)=0$ for a required row, the corresponding plug-in calculation is not available in that sample and this must be reported. Under $\pi_{x\mid\rm M}>0$, this event occurs only finitely often almost surely. Deleting a required row changes the target and is not part of the asymptotic construction.

\subsection{Implementation procedure}
\label{subsec:implementation_procedure}
The preceding formulas give the following finite-dimensional computational procedure, from observed trajectories to restricted reliability estimates, uncertainty quantification and horizon diagnostics.
\begin{center}
\fbox{%
\begin{minipage}{0.92\textwidth}
\small
\textbf{Algorithm 1. Finite-horizon reliability procedure.}

\noindent\textbf{Input.} Independent trajectories $\{Z^{(\ell)}_m:0\le m\le M,\ 1\le \ell\le L\}$, a state partition $E=U\cup D$, a finite horizon $M$, and a confidence level $1-\gamma$.

\begin{enumerate}
\item \textbf{Couple and restrict the trajectories.} Compute the backward recurrence time $U^{(\ell)}_m$ and form $X^{(\ell)}_m=(Z^{(\ell)}_m,U^{(\ell)}_m)$. Fix the finite-horizon state set and the transition rows required by the target before estimation. Record any required row with zero empirical occupation count. Report for every required row the total exposure $N_x(L)$, the number $D_x(L)$ of trajectories contributing at least one exposure, and the estimated mean exposure $N_x(L)/L$.

\item \textbf{Estimate the finite restricted chain.} From the counts $N^{(0)}_x(L)$, $N_x(L)$ and $N_{xy}(L)$, compute $\widehat\ba_{\EM}(L)$ and $\widehat{\bf P}_{\EM\EM}(L)$ as in \eqref{eq:MLE_restricted}. Form the blocks $\widehat{\bf Q}_{\rm M}=(\widehat{\bf P}_{\EM\EM})_{\mathbf U_{\rm M}\mathbf U_{\rm M}}$ and $\widehat{\bf P}_{\mathbf U_{\rm M}\mathbf D_{\rm M}}$.

\item \textbf{Evaluate finite-horizon indicators.} If $\rho(\widehat{\bf Q}_{\rm M})<1$, compute $\widehat{\bf N}_{\rm M}=(I_{\mathbf U_{\rm M}}-\widehat{\bf Q}_{\rm M})^{-1}$ and evaluate the age-restricted moments. Evaluate calendar-truncated moments from the finite sum of powers or directly from the observed truncated failure times. Compute the DTIHT sequence $\widehat r_M=(\widehat r(1;M),\ldots,\widehat r(M;M))^\top$.

\item \textbf{Compute plug-in covariance matrices.} Replace $\ba_{\EM}$, $\PM$ and $\pi_{x\mid \rm M}$ by $\widehat\ba_{\EM}(L)$, $\widehat{\bf P}_{\EM\EM}(L)$ and $N_x(L)/L$ in the covariance formulas. Build the derivative vectors or matrices for the required targets, including $d^{(k)}_M$, $d^V_M$ and $d^r_{m,M}$, and assemble the corresponding plug-in covariance matrices.

\item \textbf{Report statistical uncertainty.} Use the diagonal entries of the plug-in covariance matrices for standard errors and pointwise confidence intervals. For the whole DTIHT sequence, use the joint covariance matrix to produce simultaneous Gaussian confidence envelopes and Wald tests for linear summaries. For moderate $L$ or sparse rows, compare the Gaussian envelope with a trajectory-bootstrap envelope and report the critical values used by both procedures.

\item \textbf{Report restriction diagnostics and, when appropriate, select an analysis horizon.} Compute maximum backward-recurrence times, near-boundary-set indicators, boundary-exit summaries and, when trajectories are available up to a common maximal horizon, paired nested-horizon comparisons. If an analysis horizon is selected from a prespecified grid, use the target-specific rule of Subsection~\ref{subsec:data_driven_horizon} and record the exposure, target-scale and boundary-diagnostic thresholds. Independent sample splitting gives fixed-horizon inference after selection; alternatively, a trajectory bootstrap that repeats the complete selection step approximates the sampling distribution of the selected estimator. Interpret these quantities as diagnostics for finite-horizon restriction effects, not as estimates of the unidentified infinite tail.
\end{enumerate}

\noindent\textbf{Output.} Restricted finite-horizon reliability indicators, plug-in covariance matrices, confidence intervals or simultaneous confidence envelopes, optional Wald summaries, and boundary diagnostics documenting the sensitivity of the reported targets to the observation horizon.
\end{minipage}}
\end{center}

\section{Diagnostics for restriction effects}
\label{S:diagnostics}

The diagnostics below are not estimators of the full unbounded-horizon MTTF\@. They record interaction with the backward-recurrence-time boundary and with exits from the finite-horizon coupled state space. They are applied after the finite-horizon target and its sampling uncertainty have been reported.

\subsection{Non-identifiability of the unbounded-horizon MTTF}

Fixed-horizon observations cannot identify the unobserved tail of the semi-Markov kernel. If two kernels agree on all transitions that can affect trajectories observed up to time $M$, but differ in the later tail of an up-state sojourn distribution, then the induced laws of the observed sample may be identical for every $L$, while the full unbounded-horizon MTTF differs. The next proposition records this point by a two-state construction. It shows that the full MTTF is not a nonparametric estimand in the unrestricted fixed-horizon experiment.

\medskip
\begin{proposition}
\label{prop:finite_horizon_nonidentifiability}
Fix $M\geq1$ and let $p\in(0,1)$. There exist two discrete-time semi-Markov kernels, with the same initial law and the same finite-horizon law of the observed coupled process $((Z_m,U_m),0\leq m\leq M)$, whose full mean times to first failure are different. Consequently, the full unbounded-horizon MTTF is not identifiable from $L$ independent trajectories observed on $0,1,\ldots,M$, for any $L$, unless additional restrictions are imposed on the unobserved tail of the sojourn distribution.
\end{proposition}

\begin{proof}
Let $E=\{u,d\}$, take $\mathrm U=\{u\}$ and $\mathrm D=\{d\}$, and start from state $u$ with backward recurrence time zero. Choose two distinct integers $K_1,K_2>M$. For $a=1,2$, define a semi-Markov kernel $q^{(a)}$ by
\begin{equation*}
q^{(a)}_{u d}(1)=p,
\qquad
q^{(a)}_{u d}(K_a)=1-p,
\qquad
q^{(a)}_{d u}(1)=1,
\end{equation*}
and set all other entries equal to zero. Each kernel satisfies the normalization condition and the convention excluding self-jumps. Consider an entry into the up state at a time $s\leq M$. Under both kernels, the next transition is a failure at time $s+1$, with probability $p$, or an up-state sojourn of length $K_a$, with probability $1-p$. Since $K_a>M$, the latter alternative is observed on $0,\ldots,M$ only as remaining in the up state, with backward recurrence time increasing by one at each step; its exact duration is not reached within the observation window. The down-state sojourn is deterministic of length one in both models. Induction over successive entries into $u$ before time $M$ therefore gives the same law of $((Z_m,U_m),0\leq m\leq M)$ for $q^{(1)}$ and $q^{(2)}$, and hence the same product law for any number $L$ of independent observed trajectories.

The full first-failure time from the initial up state satisfies
\begin{equation*}
\mathbb{P}_{q^{(a)}}(T_{\mathrm D}=1)=p,
\qquad
\mathbb{P}_{q^{(a)}}(T_{\mathrm D}=K_a)=1-p.
\end{equation*}
Thus
\begin{equation*}
\mathbb E_{q^{(a)}}(T_{\mathrm D})=p+(1-p)K_a,
\end{equation*}
which differs for $a=1$ and $a=2$. Two statistically indistinguishable fixed-horizon experiments therefore lead to different full MTTFs.
\end{proof}

Proposition~\ref{prop:finite_horizon_nonidentifiability} explains why the unbounded-horizon MTTF cannot be recovered nonparametrically from fixed-horizon observations alone. It does not, however, prevent meaningful comparisons between two finite horizons. If a longer horizon $K>M$ is observed, or if a finite continuation of the model up to $K$ is specified, then both $\mathrm{MTTF}_{\rm M}$ and $\mathrm{MTTF}_{\rm K}$ are finite-dimensional reliability indicators. Their difference measures the additional expected time accounted for by extending the finite-horizon restriction from $M$ to $K$. This comparison is therefore a horizon-sensitivity diagnostic, not an estimator of the unidentified infinite tail. The following deterministic identity gives the decomposition.

\medskip
\begin{proposition}
\label{prop:nested_horizon_decomposition}
Let $K>M$ and construct the restricted up-state sets ${\bf U}_{\rm M}\subset {\bf U}_{\rm K}$ from the same semi-Markov kernel. Assume that the same initial law is used at both horizons and that it is supported on ${\bf U}_{\rm M}$; write this law as $\ba_{{\bf U}_{\rm M}}$. Assume also that the two restricted up-state matrices have spectral radii smaller than one. Put
\begin{equation*}
{\bf W}_{K\setminus M}={\bf U}_{\rm K}\setminus{\bf U}_{\rm M},
\end{equation*}
and write the up-to-up block of the $K$-horizon coupled transition matrix as
\begin{equation*}
{\bf Q}_{\rm K}=
\begin{pmatrix}
{\bf Q}_{\rm M} & {\bf R}_{{\bf U}_{\rm M},{\bf W}_{K\setminus M}}\\
{\bf R}_{{\bf W}_{K\setminus M},{\bf U}_{\rm M}} & {\bf R}_{{\bf W}_{K\setminus M},{\bf W}_{K\setminus M}}
\end{pmatrix}.
\end{equation*}
The symbols ${\bf R}_{A,B}$ denote the corresponding submatrices between the indicated parts of ${\bf U}_{\rm K}$. Let $\bm t_{K\setminus M}$ denote the coordinates, on ${\bf W}_{K\setminus M}$, of the $K$-restricted mean continuation-time vector
\begin{equation*}
({\bf I}_{{\bf U}_{\rm K}}-{\bf Q}_{\rm K})^{-1}\one_{{\bf U}_{\rm K}}.
\end{equation*}
Then
\begin{equation}
\label{eq:finite_MK_decomp}
\mathrm{MTTF}_{\rm K}
=
\mathrm{MTTF}_{\rm M}
+
\ba_{{\bf U}_{\rm M}}({\bf I}_{{\bf U}_{\rm M}}-
{\bf Q}_{\rm M})^{-1}
{\bf R}_{{\bf U}_{\rm M},{\bf W}_{K\setminus M}}
\bm t_{K\setminus M}.
\end{equation}
In particular, $\mathrm{MTTF}_{\rm M}\leq \mathrm{MTTF}_{\rm K}$.
\end{proposition}

\begin{proof}
Because the two restrictions are constructed from the same coupled chain, the top-left block of the $K$-horizon up-to-up matrix, indexed by ${\bf U}_{\rm M}$, is ${\bf Q}_{\rm M}$. Let $\bm t_{\rm M}=({\bf I}_{{\bf U}_{\rm M}}-{\bf Q}_{\rm M})^{-1}\one_{{\bf U}_{\rm M}}$ be the vector of $M$-restricted mean times to failure or truncation, for starts in ${\bf U}_{\rm M}$. Let $\bm t_{\rm K}=({\bf I}_{{\bf U}_{\rm K}}-{\bf Q}_{\rm K})^{-1}\one_{{\bf U}_{\rm K}}$ be the analogous vector at horizon $K$, and partition it as $\bm t_{\rm K}=(\bm t_{{\rm K},M}^{\top},\bm t_{K\setminus M}^{\top})^{\top}$ according to ${\bf U}_{\rm M}\cup {\bf W}_{K\setminus M}$. The first block of the linear system $\bm t_{\rm K}=\one_{{\bf U}_{\rm K}}+{\bf Q}_{\rm K}\bm t_{\rm K}$ is
\begin{equation*}
\bm t_{{\rm K},M}
=\one_{{\bf U}_{\rm M}}+{\bf Q}_{\rm M}\bm t_{{\rm K},M}
+{\bf R}_{{\bf U}_{\rm M},{\bf W}_{K\setminus M}}\bm t_{K\setminus M}.
\end{equation*}
Subtracting the equation $\bm t_{\rm M}=\one_{{\bf U}_{\rm M}}+{\bf Q}_{\rm M}\bm t_{\rm M}$ gives
\begin{equation*}
\bm t_{{\rm K},M}-\bm t_{\rm M}
=({\bf I}_{{\bf U}_{\rm M}}-{\bf Q}_{\rm M})^{-1}
{\bf R}_{{\bf U}_{\rm M},{\bf W}_{K\setminus M}}\bm t_{K\setminus M}.
\end{equation*}
Premultiplication by the common initial law $\ba_{{\bf U}_{\rm M}}$ gives \eqref{eq:finite_MK_decomp}. All entries of the matrices and continuation-time vector are non-negative, which gives the monotonicity. The same monotonicity also follows pathwise from $T_{\mathrm D,\rm M}\leq T_{\mathrm D,\rm K}$ under the natural coupling of the two restricted chains.
\end{proof}

Proposition~\ref{prop:nested_horizon_decomposition} is finite-dimensional and does not invoke an infinite-state operator. It is a diagnostic identity, not a tail estimator. The continuation term is available only when the larger horizon $K$ is observed or when an additional finite model beyond $M$ is supplied.

\subsection{A target-specific, data-driven analysis horizon}
\label{subsec:data_driven_horizon}

The observation horizon and the analysis horizon must first be distinguished. If a study records every trajectory only up to a fixed time $M$, then $M$ is a design feature and no data-driven procedure can recover information beyond it. A choice among horizons is possible only when the trajectories are available up to a common maximal horizon $M^\star$ and one considers a prespecified grid
\begin{equation*}
{\cal G}=\{M_1<\cdots<M_K=M^\star\}.
\end{equation*}
The choice is necessarily target-specific. For example, an administrative calendar horizon defines the target $T_{\mathrm D}\wedge M$ and should normally be chosen from the scientific follow-up question, whereas selection by stability is relevant when an age-restricted quantity is used as a finite approximation whose sensitivity to the represented age range must be controlled.

Let $\theta_M$ denote a scalar finite-horizon target and let ${\cal R}_M$ be its required transition rows. Before comparing target estimates, remove from consideration every horizon for which a required row has zero exposure, the estimated resolvent is not spectrally stable, or prespecified adequacy thresholds are not met. For a resolvent target, define the admissible set by
\begin{align*}
{\cal G}_{\rm adm}
=\bigl\{M_j\in{\cal G}:\;&
\min_{x\in{\cal R}_{M_j}}N_x^{(M_j)}(L)\geq n_{\min},\\
&\min_{x\in{\cal R}_{M_j}}D_x^{(M_j)}(L)\geq d_{\min},\\
&\rho(\widehat{\bf Q}_{M_j})\leq1-\delta_\rho
\bigr\}.
\end{align*}
Here $n_{\min}$, $d_{\min}$ and $\delta_\rho$ are fixed before inspecting the target estimates. These thresholds are diagnostics rather than universal constants; they should reflect the sample size and the reliability calculation under consideration.

Suppose first that ${\cal G}_{\rm adm}$ is nonempty and that $\theta_M$ is non-decreasing in $M$, as is the age-restricted mean in Proposition~\ref{prop:nested_horizon_decomposition}. Let $M_{\max}$ be the largest element of ${\cal G}_{\rm adm}$ and put
\begin{equation*}
\widehat\Delta_j
=
\widehat\theta_{M_{\max}}-\widehat\theta_{M_j}.
\end{equation*}
Because both estimates use the same trajectories, the standard error $\widehat s_j$ of $\widehat\Delta_j$ must be computed from their paired influence functions or from a trajectory bootstrap, not by treating the estimates as independent. Let $c_{1-\gamma}$ be a joint one-sided critical value for the finite vector of differences. Given a prespecified tolerance $\varepsilon_\theta>0$ on the scale of the target and a diagnostic boundary tolerance $\varepsilon_{\rm tail}>0$, define
\begin{equation*}
\widehat M
=
\min\left\{
M_j\in{\cal G}_{\rm adm}:
\widehat\Delta_j+c_{1-\gamma}\widehat s_j\leq\varepsilon_\theta,
\quad
q^{\rm diag}_{\rm tail}(L;M_j)\leq\varepsilon_{\rm tail}
\right\}.
\end{equation*}
Here $q^{\rm diag}_{\rm tail}$ is the boundary-exit diagnostic defined below; the constituent rowwise bounds use adjusted error probabilities when several boundary rows are considered. Since the diagnostic also contains plug-in transition quantities, it is used as a screening quantity rather than as an exact confidence bound. For a non-monotone target, the first condition is replaced by a simultaneous upper bound for
\begin{equation*}
\max_{M_k\in{\cal G}_{\rm adm}:M_k>M_j}
|\theta_{M_k}-\theta_{M_j}|.
\end{equation*}
If the defining set is empty, the data do not support declaring any candidate horizon adequate at the chosen tolerances; the largest admissible horizon and all diagnostics should then be reported without an adequacy claim.

This rule balances two effects that move in opposite directions. Increasing $M$ reduces age-restriction sensitivity but enlarges the coupled state space and can create sparsely exposed rows. It does not turn the finite comparison into an estimate of the unidentified infinite tail. Moreover, fixed-$M$ confidence intervals applied after selecting $\widehat M$ from the same data do not automatically retain their nominal coverage. The grid and tolerances should therefore be prespecified. Sample splitting, using one set of trajectories to select $\widehat M$ and the other for fixed-horizon inference, avoids using the same observations twice. A trajectory bootstrap that repeats the complete selection and estimation procedure in each resample approximates the sampling distribution of the selected estimator. Its validity near horizon-selection boundaries requires a separate analysis and is not established here.

\subsection{Maximum backward recurrence times and near-boundary diagnostics}

The maximum backward recurrence times and the near-boundary indicators measure contact with the boundary $M-1$ of the represented age coordinate. Repeated visits near this boundary indicate possible sensitivity of age-restricted moment calculations to the horizon.

For each state $i\in E$, define
\begin{equation*}
\widehat u_i^{\max}(L)=
\max\{u:\exists (\ell,m)\textrm{ such that }(Z_m^{(\ell)},U_m^{(\ell)})=(i,u)\},
\end{equation*}
with the convention $\widehat u_i^{\max}(L)=-1$ if state $i$ is never visited. The deterministic evolution of the backward recurrence time implies
\begin{equation*}
N_{(i,u)}(L)>0
\quad\Longrightarrow\quad
N_{(i,v)}(L)>0,
\qquad v=0,\ldots,u-1.
\end{equation*}
Thus, if $\widehat u_i^{\max}(L)$ is far below $M-1$, the data do not approach the truncation boundary in block $i$.

For a width $d\in\{1,\ldots,M\}$, define
\begin{equation*}
\partial^{(d)}\EM(i)=\{(i,u)\in\EM:u\geq M-d\}.
\end{equation*}
For trajectory $\ell$, let
\begin{equation*}
B_{i,d}^{(\ell)}=
\mathds{1}\{\exists m\leq M:(Z_m^{(\ell)},U_m^{(\ell)})\in\partial^{(d)}\EM(i)\}.
\end{equation*}
The variables $B_{i,d}^{(1)},\ldots,B_{i,d}^{(L)}$ are independent Bernoulli variables with parameter $\delta_{i,d}(\rm M)$. Let $b_{i,d}=\sum_{\ell=1}^L B_{i,d}^{(\ell)}$. The one-sided Clopper--Pearson upper confidence bound \citep{ClopperPearson1934} at error probability $\gamma$ is
\begin{equation*}
\delta_{i,d}^{\rm ub}(L;\gamma)
=
\begin{cases}
\mathrm{Beta}^{-1}\left(1-\gamma;b_{i,d}+1,L-b_{i,d}\right), & b_{i,d}<L,\\[1mm]
1, & b_{i,d}=L.
\end{cases}
\end{equation*}
If no visit to this near-boundary set is observed, this becomes
\begin{equation*}
\delta_{i,d}^{\rm ub}(L;\gamma)=1-\gamma^{1/L}.
\end{equation*}

\subsection{Boundary exit probability}

After contact with the boundary, the relevant quantity is the probability of leaving the represented age space. If the coupled process is in $x=(i,M-1)$, then a transition preserving the physical state moves the age from $M-1$ to $M$, outside $\EM$. This is not a failure event; it is exit into an unrepresented continuation region. For up states, a non-negligible boundary-exit probability indicates that age-restricted moment calculations may depend on sojourn-time continuation beyond the observation horizon.

For $x\in\EM$, define the one-step probability of leaving the finite-horizon coupled state space by
\begin{equation*}
p^{\rm out}_{x;\rm M}=1-\sum_{y\in\EM}P(x,y).
\end{equation*}
By the structure of the coupled chain, this probability can be nonzero only at boundary states $x=(i,M-1)$. The empirical exit count is
\begin{equation*}
N_x^{\rm out}(L)=N_x(L)-\sum_{y\in\EM}N_{xy}(L).
\end{equation*}
For a boundary row $x=(i,M-1)$ with $N_x(L)>0$, visits can occur only at time $M-1$. Conditional on $N_x(L)$, the exit count is therefore binomial with parameter $p^{\rm out}_{x;\rm M}$. Writing $n=N_x(L)$ and $e=N_x^{\rm out}(L)$, the one-sided Clopper--Pearson upper confidence bound is
\begin{equation*}
p^{\rm out,ub}_{x;\rm M}(L;\gamma_x)
=
\begin{cases}
\mathrm{Beta}^{-1}\left(1-\gamma_x;e+1,n-e\right), & e<n,\\[1mm]
1, & e=n.
\end{cases}
\end{equation*}
If $e=0$ and $n>0$, it reduces to
\begin{equation*}
p^{\rm out,ub}_{x;\rm M}(L;\gamma_x)=1-\gamma_x^{1/n}.
\end{equation*}
When bounds are reported for several boundary rows, simultaneous rowwise coverage at least $1-\gamma$ is obtained by choosing the error probabilities so that $\sum_x\gamma_x\leq\gamma$.

For MTTF, a useful diagnostic summary is the finite-chain probability of entering unobserved up-state backward recurrence times before entering a down state. With a vector $\bm p^{\rm out,ub}_{{\bf U}_{\rm M}}$ collecting the boundary exit probability bounds on ${\bf U}_{\rm M}$ and zeros elsewhere, define
\begin{equation*}
q_{\rm tail}^{\rm diag}(L;\rm M)
=(\widehat\ba_{\EM}(L))_{{\bf U}_{\rm M}}
\left\{{\bf I}_{{\bf U}_{\rm M}}-(\PMh(L))_{{\bf U}_{\rm M}{\bf U}_{\rm M}}\right\}^{-1}
\bm p^{\rm out,ub}_{{\bf U}_{\rm M}}.
\end{equation*}
The quantity $q_{\rm tail}^{\rm diag}(L;\rm M)$ is defined when the displayed inverse exists and summarizes the potential contribution of the boundary-exit probabilities under the fitted finite-horizon chain. It should be read only as a boundary-exit diagnostic: it does not quantify the duration of the unobserved continuation. Because it combines separate row-wise upper bounds with plug-in estimates of $\widehat\ba_{\EM}(L)$ and $\PMh(L)$, it is neither an exact simultaneous confidence bound nor necessarily bounded by one.


\section{Numerical results}
\label{S:application}

The numerical section has two parts. The simulated alternating-renewal model has known transition probabilities and is used to check the finite-horizon formulas, sampling approximations and diagnostics. The EBMT example uses reconstructed post-transplant trajectories from the \texttt{mstate} package and reports calendar-truncated summaries and hitting intensities for a complete-case subset.

\subsection{Simulated semi-Markov system}

The numerical calculations were carried out in Python using NumPy, SciPy and Numba. The single simulated sample contains $L=5000$ trajectories observed up to time $50$ and uses seed $314159$. The plug-in simultaneous DTIHT critical value is based on $250000$ Gaussian draws with seed $271828$. The Monte Carlo settings for the remaining experiments are stated with the corresponding tables.

We first consider a simulated semi-Markov system where $E=\{1,2\}$, $\mathrm U=\{1\}$ and $\mathrm D=\{2\}$. The process alternates between the two states. Thus, from a coupled state $(i,u)$, the next state is either $(i,u+1)$ or $(3-i,0)$. The initial condition is $(Z_0,U_0)=(1,0)$ and
\begin{equation*}
p_{1,u;1}=\begin{cases}
0.98, & 0\leq u\leq9,\\
0.93, & 10\leq u\leq24,\\
0.75, & u\geq25,
\end{cases}
\qquad
p_{2,u;2}=\begin{cases}
0.90, & 0\leq u\leq4,\\
0.60, & u\geq5.
\end{cases}
\end{equation*}

\subsubsection{Restricted MTTF and moment characteristics}

In this alternating model there is only one operational state before first failure. Hence $T_{\mathrm D,\rm M}=T_{\mathrm D}^{[M]}=T_{\mathrm D}\wedge M$, and the unbounded-horizon time to failure is the length of the initial up-state sojourn. Hence the theoretical unbounded-horizon MTTF is available from the known transition probabilities. With the usual empty-product convention,
\begin{equation*}
\mathrm{MTTF}=\sum_{n\geq 0}\mathbb{P}(T_{\mathrm D}>n)
=\sum_{n\geq 0}\prod_{u=0}^{n-1}p_{1,u;1}
=17.989,
\end{equation*}
up to the displayed precision. This value is used only as a benchmark for the finite-horizon restriction effect. Before introducing simulated samples, we first examine how the theoretical restricted mean changes with the horizon. Table~\ref{tab:horizon_sensitivity_example} reports $\mathrm{MTTF}_{\rm M}$, the tail probability $\mathbb{P}(T_{\mathrm D}>M)$, and the truncation gap $\Delta_{\rm M}=\mathrm{MTTF}-\mathrm{MTTF}_{\rm M}$. The column $\mathrm{se}_{5000}$ is the first-order model-based standard error of $\widehat{\mathrm{MTTF}}_{\rm M}$ for $L=5000$, computed from Theorem~\ref{thm:CLT_MkTTF_restricted}. The ratio $\Delta_{\rm M}/\mathrm{se}_{5000}$ compares the finite-horizon truncation gap with the sampling error expected at this illustrative sample size.

\begin{table}[htbp]
\centering
\caption{Horizon sensitivity of $\mathrm{MTTF}_{\rm M}$.}
\label{tab:horizon_sensitivity_example}
\begin{tabular}{c|ccccc}
\hline
$M$ & $\mathrm{MTTF}_{\rm M}$ & $\mathbb{P}(T_{\mathrm D}>M)$ & $\Delta_{\rm M}$ & $\mathrm{se}_{5000}$ & $\Delta_{\rm M}/\mathrm{se}_{5000}$\\
\hline
10 & 9.146 & 0.81707 & 8.843 & 0.030 & 292.7\\
15 & 12.698 & 0.56843 & 5.291 & 0.052 & 101.7\\
20 & 15.170 & 0.39545 & 2.820 & 0.073 & 38.4\\
25 & 16.889 & 0.27511 & 1.100 & 0.092 & 12.0\\
30 & 17.728 & 0.06528 & 0.261 & 0.100 & 2.6\\
40 & 17.974 & 0.00368 & 0.015 & 0.093 & 0.16\\
50 & 17.988 & 0.00021 & 0.001 & 0.084 & 0.01\\
\hline
\end{tabular}
\end{table}

These are theoretical calculations for the specified semi-Markov model, not estimates from the sample. At $M=20$, the truncation gap is about $38$ times the sampling standard error for $L=5000$. A confidence interval centred on $\widehat{\mathrm{MTTF}}_{20}$ can therefore have good coverage for $\mathrm{MTTF}_{20}$ while being far from the full MTTF\@. At $M=50$, the corresponding gap is negligible on the same sampling scale.

For the two horizons retained in the single-sample analysis, Table~\ref{tab:theoretical_restricted_shape} reports the theoretical restricted moment characteristics obtained from \eqref{eq:MTTF_M_population} and from the first four factorial moments in \eqref{eq:MkTTF_M_population}. These quantities are computed from the known transition probabilities and therefore serve as benchmarks for the plug-in estimates. The table gives the restricted mean, the restricted standard deviation $s_{\rm M}$, the restricted coefficient of skewness $\gamma_{1,\rm M}$, the restricted coefficient of kurtosis $\gamma_{2,\rm M}$, and the tail probability $\mathbb{P}(T_{\mathrm D}>M)$, reported as a reference value for the finite-horizon restriction. The contrast between the two horizons is visible beyond the mean: at $M=20$ the distribution is strongly left-skewed because many paths reach the boundary before failure, whereas at $M=50$ the restricted distribution is very close to the time-to-failure distribution for the present parameters.

\begin{table}[htbp]
\centering
\caption{Theoretical restricted moment characteristics of $T_{\mathrm D,\rm M}$ for the simulated semi-Markov system.}
\label{tab:theoretical_restricted_shape}
\begin{tabular}{c|ccccc}
\hline
$M$ & $\mathrm{MTTF}_{\rm M}$ & $s_{\rm M}$ & $\gamma_{1,\rm M}$ & $\gamma_{2,\rm M}$ & $\mathbb{P}(T_{\mathrm D}>M)$\\
\hline
20 & 15.170 & 5.643 & -0.967 & 2.821 & 0.39545\\
50 & 17.988 & 8.825 & 0.065 & 2.336 & 0.00021\\
\hline
\end{tabular}
\end{table}

We now consider one simulated data set of $L=5000$ independent trajectories generated up to time $50$. The results for $M=20$ are obtained by restricting these same trajectories to the time interval $0,\ldots,20$, whereas the results for $M=50$ use the complete simulated paths. This construction compares two finite-horizon analyses on the basis of the same underlying sample. We first report plug-in estimates from this single data set; repeated-sampling assessments are considered afterwards.

The first four restricted factorial moments provide, through Theorem~\ref{thm:CLT_factorial_vector_restricted}, a joint asymptotic distribution for the corresponding moment characteristics. Table~\ref{tab:joint_shape_example} compares the theoretical targets from Table~\ref{tab:theoretical_restricted_shape} with the single-sample plug-in estimates of the restricted mean, standard deviation, skewness and kurtosis. The last column gives the corresponding first-order standard errors.

\begin{table}[htbp]
\centering
\caption{Joint moment characteristics obtained from the first four restricted factorial moments. The reported standard errors are computed from the plug-in first-order covariance formula.}
\label{tab:joint_shape_example}
\begin{tabular}{cc|ccc}
\hline
$M$ & characteristic & target & estimate & s.e.\\
\hline
20 & mean & 15.170 & 15.205 & 0.073\\
20 & sd & 5.643 & 5.610 & 0.048\\
20 & skewness & -0.967 & -0.978 & 0.024\\
20 & kurtosis & 2.821 & 2.843 & 0.060\\
\hline
50 & mean & 17.988 & 18.132 & 0.084\\
50 & sd & 8.825 & 8.779 & 0.053\\
50 & skewness & 0.065 & 0.064 & 0.021\\
50 & kurtosis & 2.336 & 2.354 & 0.045\\
\hline
\end{tabular}
\end{table}

To separate the single-sample illustration from repeated-sampling performance, we ran a finite-sample coverage assessment for the matrix plug-in estimator of the restricted MTTF\@. For each pair $(M,L)$, $10^4$ independent Monte Carlo replications were generated from the same alternating-renewal model. Each replication estimates the age-dependent transition probabilities from all observed up-state exposures and then evaluates the matrix plug-in functional in Theorem~\ref{thm:CLT_MkTTF_restricted}. The random-number stream uses seed $20260619+1000M+L$. The Gaussian $95\%$ confidence interval is
\begin{equation*}
\widehat{\mathrm{MTTF}}_{\rm M}(L)
\pm 1.96\,\widehat{\mathrm{se}}\{\widehat{\mathrm{MTTF}}_{\rm M}(L)\}.
\end{equation*}
The target is the theoretical restricted mean $\mathrm{MTTF}_{\rm M}$, not the full unbounded-horizon MTTF\@. Coverage denotes the empirical proportion of confidence intervals that contain this target over the Monte Carlo replications. Table~\ref{tab:mttf_coverage_assessment} reports the average estimate, the Monte Carlo standard deviation (MC sd), the average plug-in standard error (mean s.e.) and the corresponding empirical coverage.

\begin{table}[htbp]
\centering
\caption{Monte Carlo coverage of Gaussian $95\%$ confidence intervals for $\mathrm{MTTF}_{\rm M}$. The coverage column gives the empirical proportion of intervals containing the theoretical restricted mean.}
\label{tab:mttf_coverage_assessment}
\begin{tabular}{cc|ccccc}
\hline
$M$ & $L$ & target & mean estimate & MC sd & mean s.e. & coverage\\
\hline
20 & 100  & 15.170 & 15.174 & 0.520 & 0.516 & 0.942\\
20 & 500  & 15.170 & 15.166 & 0.233 & 0.232 & 0.949\\
20 & 1000 & 15.170 & 15.170 & 0.164 & 0.164 & 0.949\\
\hline
50 & 100  & 17.988 & 17.992 & 0.585 & 0.588 & 0.950\\
50 & 500  & 17.988 & 17.997 & 0.264 & 0.264 & 0.950\\
50 & 1000 & 17.988 & 17.987 & 0.188 & 0.187 & 0.948\\
\hline
\end{tabular}
\end{table}

Table~\ref{tab:mttf_coverage_assessment} shows that the plug-in standard errors reproduce the Monte Carlo standard deviations and that the Gaussian intervals are close to the $95\%$ target even for moderate $L$. With $10^4$ Monte Carlo replications, the Monte Carlo standard error of a coverage estimate near $0.95$ is about $0.002$. Thus, in this parameter setting, the first-order asymptotic approximation provides an accurate scalar inference procedure for the restricted MTTF. The coverage assessment is nevertheless conditional on the finite-horizon target: for $M=20$ the interval covers $\mathrm{MTTF}_{20}$, while Table~\ref{tab:horizon_sensitivity_example} shows that this restricted mean should not be interpreted as the full MTTF.

We finally assess the finite-order approximation for the restricted MTTF\@. The reference experiment uses $M=50$, $L=30$ trajectories per replication and $10^5$ independent replications of the whole experiment. The first-order and second-order approximations are those described in Subsection~\ref{subsec:higher_order_restricted}. More precisely, the first-order interval is obtained from the first differential in \eqref{eq:DMTTF_restricted} together with Theorem~\ref{thm:CLT_MkTTF_restricted} with $k=1$, whereas the second-order interval adds the $L^{-1}$ term in \eqref{eq:MTTF_first_second_approx_restricted} by simulating $10^6$ copies of the limiting Gaussian matrix. For this assessment, the Monte Carlo skewness of $\widehat{\mathrm{MTTF}}_{50}(30)$ is $0.157$, so that a symmetric normal approximation is not expected to reproduce both endpoints exactly. In Table~\ref{tab:mttf_second_order_mc}, $q^{\rm MC}_{\alpha}$ denotes the empirical
$\alpha$-quantile of $\widehat{\mathrm{MTTF}}_{50}(30)$ computed from the
$10^5$ Monte Carlo replications. For each approximation method,
$q^{\rm app}_{\alpha}$ denotes the corresponding $\alpha$-quantile obtained from the
first-order or second-order approximating distribution. The quantities
\begin{equation*}
\Delta_L=q^{\rm app}_{0.025}-q^{\rm MC}_{0.025},
\qquad
\Delta_U=q^{\rm app}_{0.975}-q^{\rm MC}_{0.975}.
\end{equation*}
are the signed errors of the lower and upper endpoints of the central $95\%$ interval.
Values of $\Delta_L$ and $\Delta_U$ closer to zero indicate a better approximation
of the Monte Carlo reference interval.

\begin{table}[htbp]
\centering
\caption{Central $95\%$ sampling-distribution interval for $\widehat{\mathrm{MTTF}}_{50}(30)$ and first-order and curvature-adjusted Gaussian approximations. The Monte Carlo reference uses $10^5$ independent experiments; the curvature-adjusted approximation uses $10^6$ Gaussian draws.}
\label{tab:mttf_second_order_mc}
\begin{tabular}{c|ccccc}
\hline
method & $2.5\%$ & $50\%$ & $97.5\%$ & $\Delta_L$ & $\Delta_U$\\
\hline
Monte Carlo & 15.981 & 17.990 & 20.209 & -- & --\\
first order & 15.874 & 17.988 & 20.103 & -0.108 & -0.106\\
curvature adjusted & 16.013 & 17.924 & 20.313 & 0.032 & 0.105\\
\hline
\end{tabular}
\end{table}

The quadratic approximation is a curvature-adjusted Gaussian approximation, not a complete order-$L^{-1}$ bias correction and not an Edgeworth expansion. The ratio-bias contribution in Remark~\ref{rem:second_order_bias_MLE} is omitted. In Table~\ref{tab:mttf_second_order_mc}, the total absolute endpoint discrepancy decreases from $0.214$ to $0.137$. The absolute lower-endpoint discrepancy decreases from $0.108$ to $0.032$, whereas the upper endpoint discrepancy changes little in absolute value and the approximating median moves farther from the Monte Carlo median. The refinement is therefore not uniformly superior over all distributional summaries.

The restriction diagnostics of Section~\ref{S:diagnostics} were also computed from the single sample of $5000$ trajectories. The near-boundary strip contains the last five up-state ages and the one-sided bounds use error probability $0.05$. Table~\ref{tab:simulation_diagnostics} reports pointwise bounds for the displayed up-state row.

\begin{table}[htbp]
\centering
\caption{Restriction diagnostics for the simulated sample with $L=5000$, strip width $d=5$ and one-sided confidence level $95\%$.}
\label{tab:simulation_diagnostics}
\setlength{\tabcolsep}{3.5pt}
\begin{tabular}{c|rrrrrrr}
\hline
$M$ & $\widehat u_1^{\max}$ & strip hits & $\delta_{1,5}^{\rm ub}$ & $N_{(1,M-1)}$ & $N_{(1,M-1)}^{\rm out}$ & $p_{(1,M-1)}^{\rm out,ub}$ & $q_{\rm tail}^{\rm diag}$\\
\hline
20 & 19 & 2925 & 0.5965 & 2124 & 1981 & 0.9414 & 0.3992\\
50 & 49 & 5    & 0.0021 & 1    & 1    & 1.0000 & 0.0002\\
\hline
\end{tabular}
\end{table}

At $M=20$, visits to the boundary strip and exits from the represented up-state ages are frequent. At $M=50$, the boundary row is reached only once. Its rowwise upper bound is therefore uninformative, but the fitted probability of reaching an unobserved up age before failure is approximately $2\times10^{-4}$. The accessibility and row-exit diagnostics must consequently be interpreted together.

\subsubsection{Discrete-time intensity of the hitting time}

The MTTF analysis above concerns the time to the first entrance into the failed state. We now use the same alternating semi-Markov model in its repairable interpretation, where returns from the failed state to the operational state allow several up-to-down transitions along a trajectory. In this setting, the DTIHT describes the time sequence of failure-occurrence probabilities. We first consider inference for selected time points of the DTIHT. In Table~\ref{tab:intensity_envelopes_example}, $r(m)$ denotes the theoretical model value, whereas $\widehat r(m;20)$ and $\widehat r(m;50)$ are the plug-in estimates obtained from the same simulated data set under the two finite horizons. The column ``pointwise h.w.'' is
\begin{equation*}
z_{0.975}\left\{\widehat\Gamma^r_{\rm M}(L)_{mm}/L\right\}^{1/2},
\end{equation*}
whereas the column ``simultaneous h.w.'' is
\begin{equation*}
c_{0.95}\left\{\widehat\Gamma^r_{\rm M}(L)_{mm}/L\right\}^{1/2},
\end{equation*}
with $c_{0.95}$ the $0.95$ quantile of the maximum of the standardized Gaussian vector in Theorem~\ref{thm:CLT_r_curve_restricted}. Thus the first quantity is the marginal $95\%$ half-width for one fixed ordinate and the second is the simultaneous $95\%$ half-width for the entire sequence $m=1,\ldots,20$. Both half-width columns are computed from the $M=20$ plug-in covariance matrix and refer to $\widehat r(m;20)$; the $M=50$ estimates are included only for the nested-horizon comparison.

\begin{table}[htbp]
\centering
\caption{Selected DTIHT ordinates and Gaussian $95\%$ half-widths for one simulated data set. The half-widths use the $M=20$ plug-in covariance matrix.}
\label{tab:intensity_envelopes_example}
\begin{tabular}{c|ccccc}
\hline
$m$ & $r(m)$ & $\widehat r(m;20)$ & $\widehat r(m;50)$ & pointwise h.w. & simultaneous h.w.\\
\hline
1 & 0.0200 & 0.0184 & 0.0182 & 0.0032 & 0.0049\\
5 & 0.0187 & 0.0172 & 0.0162 & 0.0033 & 0.0050\\
10 & 0.0180 & 0.0133 & 0.0154 & 0.0029 & 0.0045\\
15 & 0.0466 & 0.0476 & 0.0448 & 0.0056 & 0.0086\\
20 & 0.0388 & 0.0374 & 0.0376 & 0.0046 & 0.0070\\
\hline
\end{tabular}
\end{table}

We then assessed the repeated-sampling behaviour of the DTIHT estimator. Table~\ref{tab:intensity_mc_assessment} is based on $1000$ independent Monte Carlo replications with $M=20$ and $L=5000$. For each replication, the finite-dimensional DTIHT estimator and the plug-in pointwise standard errors were recomputed from the estimated restricted transition matrix. The columns report the theoretical value $r(m)$, the average estimate, the Monte Carlo standard deviation, the average plug-in standard error and the pointwise coverage. Here pointwise coverage is the empirical proportion of confidence intervals containing $r(m)$ at the fixed time point $m$. The Monte Carlo standard error of a coverage entry near $0.95$ is about $0.007$.

\begin{table}[htbp]
\centering
\caption{Monte Carlo assessment for selected DTIHT ordinates, with $M=20$ and $L=5000$.}
\label{tab:intensity_mc_assessment}
\begin{tabular}{c|ccccc}
\hline
$m$ & $r(m)$ & mean estimate & MC sd & mean s.e. & pointwise coverage\\
\hline
1  & 0.0200 & 0.0199 & 0.0017 & 0.0017 & 0.952\\
5  & 0.0187 & 0.0187 & 0.0018 & 0.0017 & 0.949\\
10 & 0.0180 & 0.0181 & 0.0017 & 0.0017 & 0.960\\
15 & 0.0466 & 0.0465 & 0.0029 & 0.0028 & 0.945\\
20 & 0.0388 & 0.0389 & 0.0024 & 0.0024 & 0.948\\
\hline
\end{tabular}
\end{table}

The agreement between the empirical Monte Carlo standard deviations and the average plug-in standard errors supports the covariance calculation used for the finite-dimensional DTIHT sequence. The coverage column is pointwise, because each time point is assessed separately. When the whole sequence is reported at once, the dependence between ordinates must be taken into account.

Figure~\ref{fig:dtiht_envelopes_M20} is the graphical counterpart of Table~\ref{tab:intensity_envelopes_example}: the table reports selected ordinates, whereas the figure displays the full finite-dimensional DTIHT sequence for $m=1,\ldots,20$. It includes pointwise confidence intervals and simultaneous Gaussian confidence envelopes; the latter are wider because they use the joint Gaussian critical value for the whole sequence.

\begin{figure}[!htbp]
\centering
\includegraphics[width=0.88\textwidth]{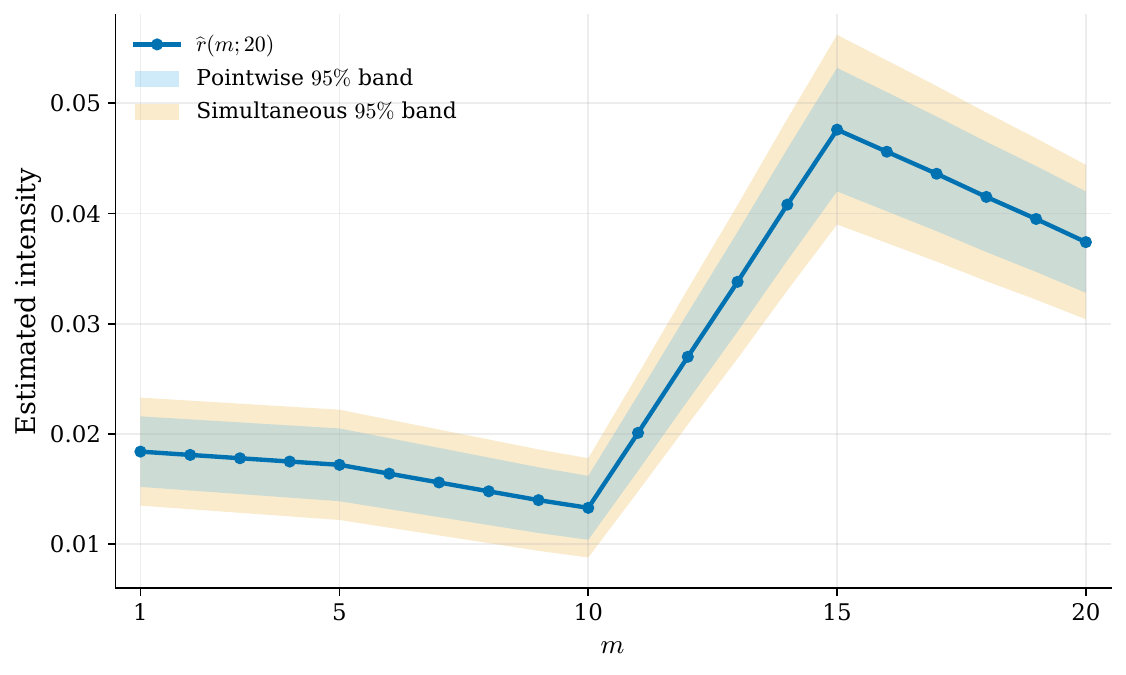}
\caption{Discrete-time intensity of the hitting time for the simulated data set with $M=20$. The line gives the estimated intensity, and the shaded regions give pointwise and simultaneous Gaussian $95\%$ confidence envelopes.}
\label{fig:dtiht_envelopes_M20}
\end{figure}

\subsection{EBMT post-transplant multi-state trajectories}
\label{subsec:ebmt}

We consider an empirical multiple-trajectory example based on the \texttt{ebmt3} data set from the European Group for Blood and Marrow Transplantation (EBMT), available in the \texttt{mstate} package; see \citet{Putter2007} and \citet{deWreedeFioccoPutter2011}. The data contain $2204$ independent post-transplant patient histories and record the time to platelet recovery and the time to relapse or death. The usual multi-state representation has three states,
\begin{equation*}
E=\{1,2,3\}=\{\mathrm{Tx},\mathrm{PR},\mathrm{RD}\}.
\end{equation*}
where $\mathrm{Tx}$ denotes the post-transplant state before platelet recovery, $\mathrm{PR}$ denotes platelet recovery, and $\mathrm{RD}$ denotes relapse or death. The possible transitions are
\begin{equation*}
\mathrm{Tx}\longrightarrow \mathrm{PR},\qquad
\mathrm{Tx}\longrightarrow \mathrm{RD},\qquad
\mathrm{PR}\longrightarrow \mathrm{RD}.
\end{equation*}
Thus every patient starts in state $\mathrm{Tx}$ at time zero and contributes one trajectory. For the reliability interpretation used here, we set
\begin{equation*}
U=\{\mathrm{Tx},\mathrm{PR}\},\qquad D=\{\mathrm{RD}\}.
\end{equation*}
The absorbing state $\mathrm{RD}$ is the analogue of a failed state; consequently, the discrete-time intensity of the hitting time coincides with the finite-horizon probability mass function of the first entrance into $D$. The terminology is reliability terminology applied to a multi-state event-history data set, and no clinical conclusion is drawn from the numerical values.

The data are reported in days. Because the present paper is formulated in discrete time, a transition observed at continuous time $t$ is assigned to the discrete day $\lceil t\rceil$. For a fixed horizon $M$, measured in days, define
\begin{equation*}
T_{\mathrm{PR},\ell}=\begin{cases}
\lceil \texttt{prtime}_{\ell}\rceil, & \text{if }\texttt{prstat}_{\ell}=1,\\
+\infty, & \text{if }\texttt{prstat}_{\ell}=0,
\end{cases}
\qquad
T_{\mathrm{RD},\ell}=\begin{cases}
\lceil \texttt{rfstime}_{\ell}\rceil, & \text{if }\texttt{rfsstat}_{\ell}=1,\\
+\infty, & \text{if }\texttt{rfsstat}_{\ell}=0.
\end{cases}
\end{equation*}
A patient censored before $M$ and before reaching $\mathrm{RD}$ does not provide a complete trajectory on $\{0,\ldots,M\}$ and is omitted from the complete-case analysis. To make nested-horizon comparisons meaningful, we fix a maximum analysis horizon $M^\star=1825$ days, retain the $L^\star=1170$ patients with relapse/death observed before $M^\star$ or follow-up extending beyond $M^\star$, and then evaluate all horizons $M\leq M^\star$ on this same retained set.

For a retained trajectory, the reconstructed discrete path is
\begin{equation*}
Z_m^{(\ell)}=
\begin{cases}
\mathrm{Tx}, & 0\leq m<T_{\mathrm{PR},\ell}\wedge T_{\mathrm{RD},\ell},\\
\mathrm{PR}, & T_{\mathrm{PR},\ell}\leq m<T_{\mathrm{RD},\ell}\text{ and }T_{\mathrm{PR},\ell}<T_{\mathrm{RD},\ell},\\
\mathrm{RD}, & T_{\mathrm{RD},\ell}\leq m,
\end{cases}
\qquad 0\leq m\leq M^\star.
\end{equation*}
The backward recurrence time is then computed by resetting it to zero at each observed transition. Same-day events, if present, are assigned according to the convention that $\mathrm{RD}$ has priority over $\mathrm{PR}$, because the target is the first entrance into the failed set. The coupled observations are
\begin{equation*}
X_m^{(\ell)}=(Z_m^{(\ell)},U_m^{(\ell)}),\qquad \ell=1,\ldots,L^\star,
\quad m=0,\ldots,M^\star.
\end{equation*}
The analysis below is conditional on inclusion in the $1170$-patient complete-case subset. It is therefore descriptive of this selected population and is not an estimate for the original cohort of $2204$ patients without a censoring adjustment. Since the true semi-Markov law is unknown, uncertainty is assessed with $B=10000$ nonparametric trajectory-bootstrap resamples of the retained histories. The bootstrap standard errors and percentile intervals quantify sampling variability within this complete-case analysis; they are not a substitute for independent-censoring theory.

The inclusion event depends jointly on follow-up and the observed event history. Unless censoring is independent of the multi-state process under conditions strong enough to justify complete-case conditioning, the retained patients can differ systematically from those censored before $M^{\star}$, and the resulting selection bias cannot be removed by increasing $L^{\star}$. An extension to the original cohort could be based on inverse probability of censoring weights. If $C_{\ell}$ is the censoring time and $G(m+1\mid\mathcal H_{\ell m})$ is the conditional probability of remaining uncensored through the transition $m\to m+1$, weighted counts would take the form
\begin{align*}
N_x^{w}(L)
&=
\sum_{\ell=1}^{L}\sum_{m=0}^{M-1}
\frac{\mathds{1}_{\{C_{\ell}\geq m+1\}}}{\widehat G(m+1\mid\mathcal H_{\ell m})}
\mathds{1}_{\{X_m^{(\ell)}=x\}},\\
N_{xy}^{w}(L)
&=
\sum_{\ell=1}^{L}\sum_{m=0}^{M-1}
\frac{\mathds{1}_{\{C_{\ell}\geq m+1\}}}{\widehat G(m+1\mid\mathcal H_{\ell m})}
\mathds{1}_{\{X_m^{(\ell)}=x,X_{m+1}^{(\ell)}=y\}}.
\end{align*}
Under conditional independent censoring and positivity, their ratio gives a natural weighted transition estimator. Its asymptotic covariance must include both the weighted transition score and estimation of the censoring distribution; this requires a separate estimating-equation or martingale analysis and is not claimed here. The complete-case results below are therefore intentionally labelled as conditional descriptive summaries.

Table~\ref{tab:ebmt_transition_summary} first reports the transitions of the full \texttt{ebmt3} data in the three-state multi-state representation. These entries are not the finite-horizon coupled-chain estimates themselves; they are included to show that the data set contains many independent trajectories with observed intermediate-state transitions. In Tables~\ref{tab:ebmt_transition_summary} and~\ref{tab:ebmt_fixed_transition_summary}, the column ``Time in state'' denotes the total accumulated time, in years, spent in the current state by the contributing individuals before an outgoing transition, censoring or the analysis horizon. The column ``Individuals'' gives the number of distinct patients contributing to the corresponding current-state row; for instance, only patients who reach $\mathrm{PR}$ contribute to the $\mathrm{PR}$ row.

\begin{table}[htbp]
\centering
\caption{Transition summary for the EBMT \texttt{ebmt3} trajectories in the three-state representation.}
\label{tab:ebmt_transition_summary}
\begin{tabular}{lrrrrr}
\hline
Current state & Remaining/censored & To PR & To RD & Time in state & Individuals \\
\hline
Tx    & 577  & 1169 & 458 & 2439.43 & 2204 \\
PR    & 786  & 0    & 383 & 3173.03 & 1169 \\
\hline
Total & 1363 & 1169 & 841 & 5612.46 & 2204 \\
\hline
\end{tabular}
\end{table}

For the $M^\star=1825$ day complete-case subset, the corresponding finite-window transition summary is shown in Table~\ref{tab:ebmt_fixed_transition_summary}. The $1170$ retained paths contain $581$ observed transitions from $\mathrm{Tx}$ to $\mathrm{PR}$, $451$ direct transitions from $\mathrm{Tx}$ to $\mathrm{RD}$, and $373$ transitions from $\mathrm{PR}$ to $\mathrm{RD}$ before the five-year horizon. At the horizon, $138$ retained patients remain in $\mathrm{Tx}$ and $208$ remain in $\mathrm{PR}$ without entering $\mathrm{RD}$.

\begin{table}[htbp]
\centering
\caption{Transition summary for the EBMT complete-case subset retained up to $M^\star=1825$ days.}
\label{tab:ebmt_fixed_transition_summary}
\begin{tabular}{lrrrrr}
\hline
Current state & Remaining at $M^\star$ & To PR & To RD & Time in state & Individuals \\
\hline
Tx & 138 & 581 & 451 & 1078.20 & 1170 \\
PR & 208 & 0   & 373 & 1366.97 & 581 \\
\hline
Total & 346 & 581 & 824 & 2445.17 & 1170 \\
\hline
\end{tabular}
\end{table}

The EBMT summaries use calendar truncation. Put
\begin{equation*}
T_{\mathrm D}^{[M]}=T_{\mathrm{RD}}\wedge M,
\qquad
\mathrm{MTTF}^{[M]}=\EE(T_{\mathrm D}^{[M]}).
\end{equation*}
This variable must be distinguished from the age-restricted failure-or-exit time $T_{\mathrm D,\rm M}$ of Definition~\ref{def:restricted_TTF}. In the EBMT model, the transition from $\mathrm{Tx}$ to $\mathrm{PR}$ resets the backward recurrence time, so the two variables need not coincide. The direct empirical estimator is
\begin{equation*}
\widehat{\mathrm{MTTF}}_{\rm emp}^{[M]}
=
\frac{1}{L^\star}\sum_{\ell=1}^{L^\star}
\min(T_{\mathrm{RD},\ell},M).
\end{equation*}
Table~\ref{tab:ebmt_mttf_bootstrap} reports this calendar-truncated mean, its trajectory-bootstrap standard error and percentile interval, the observed probability of entering $\mathrm{RD}$ by $M$, and the calendar-horizon survival mass $\widehat{\mathbb{P}}\,(T_{\mathrm D}>M)$.

\begin{table}[htbp]
\centering
\caption{EBMT calendar-truncated mean time to relapse/death on the $M^\star=1825$ day complete-case subset. Bootstrap standard errors and percentile intervals are based on $10000$ trajectory resamples.}
\label{tab:ebmt_mttf_bootstrap}
\begin{tabular}{r r r r r r r}
\hline
$M$ & $L^\star$ & $\widehat{\mathrm{MTTF}}_{\rm emp}^{[M]}$ & boot. s.e. & 95\% boot. CI & $\widehat{\mathbb{P}}\,(T_{\mathrm D}\leq M)$ & $\widehat{\mathbb{P}}\,(T_{\mathrm D}>M)$ \\
\hline
365  & 1170 & 251.582 & 3.743  & [244.198, 258.851] & 0.516 & 0.484 \\
730  & 1170 & 404.356 & 8.320  & [387.875, 420.659] & 0.618 & 0.382 \\
1095 & 1170 & 534.591 & 12.882 & [509.234, 559.917] & 0.662 & 0.338 \\
1460 & 1170 & 652.711 & 17.382 & [618.534, 686.823] & 0.688 & 0.312 \\
1825 & 1170 & 763.331 & 21.838 & [720.507, 806.251] & 0.704 & 0.296 \\
\hline
\end{tabular}
\end{table}

The calendar-truncated mean increases with $M$ because the same retained trajectories are used at all horizons. At $M=1825$, $29.6\%$ of the retained paths have not entered $\mathrm{RD}$. Thus the five-year value is a calendar-truncated mean and should not be interpreted as an extrapolated unbounded mean time to relapse/death.

Figure~\ref{fig:ebmt_mttf_boundary} displays the same horizon comparison. The calendar-truncated mean rises with $M$, while the calendar-horizon survival mass $\widehat{\mathbb{P}}\,(T_{\mathrm D}>M)$ decreases but remains non-negligible over the five-year window.

\begin{figure}[!htbp]
\centering
\includegraphics[width=0.88\textwidth]{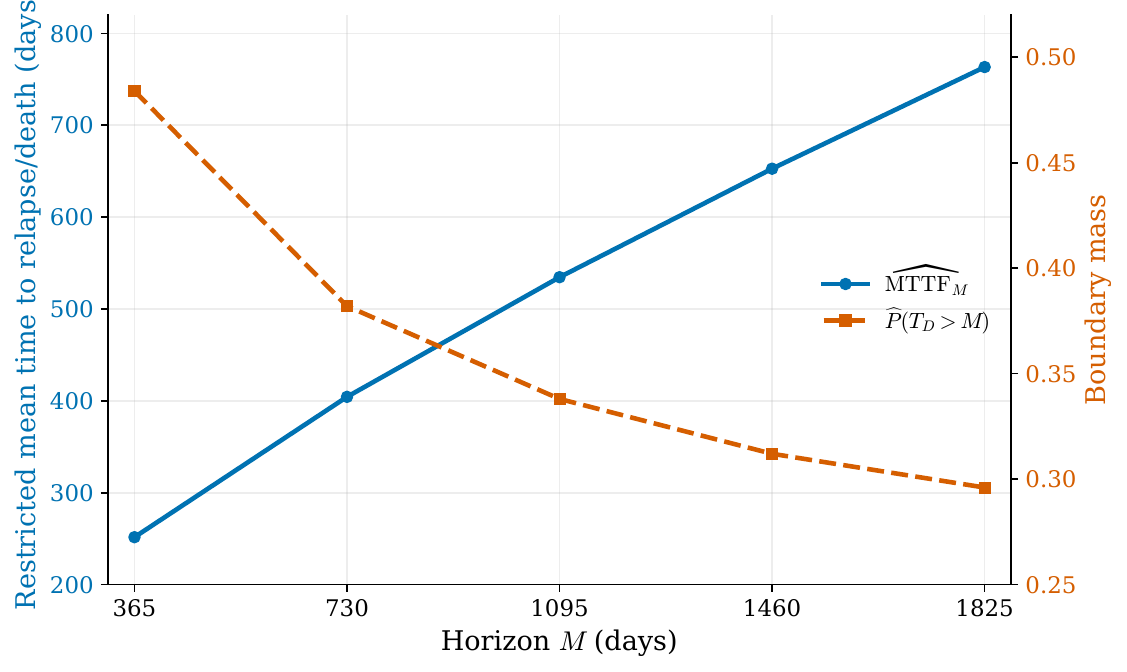}
\caption{EBMT calendar-horizon summaries on the retained $M^\star=1825$ day sample. The solid curve gives the calendar-truncated mean time to relapse/death and the dashed curve gives $\widehat{\mathbb{P}}\,(T_{\mathrm D}>M)$.}
\label{fig:ebmt_mttf_boundary}
\end{figure}

Table~\ref{tab:ebmt_moment_characteristics} gives the moment characteristics of $T_{\mathrm D}^{[M]}$ for three representative horizons. The skewness changes from negative at one year to positive at five years as the point mass at the calendar horizon decreases. The standard errors again come from the $10000$ trajectory-bootstrap resamples.

\begin{table}[htbp]
\centering
\caption{EBMT moment characteristics of the calendar-truncated time $T_{\mathrm D}^{[M]}$ on the retained $M^\star=1825$ day sample.}
\label{tab:ebmt_moment_characteristics}
\begin{tabular}{r l r r r}
\hline
$M$ & characteristic & estimate & boot. s.e. & 95\% boot. CI \\
\hline
365 & mean     & 251.582 & 3.743 & [244.198, 258.851] \\
365 & s.d.     & 129.090 & 1.475 & [126.058, 131.886] \\
365 & skewness & -0.526 & 0.054 & [-0.634, -0.423] \\
365 & kurtosis & 1.610 & 0.063 & [1.501, 1.749] \\
\hline
730 & mean     & 404.356 & 8.320 & [387.875, 420.659] \\
730 & s.d.     & 286.166 & 2.131 & [281.834, 290.202] \\
730 & skewness & 0.060 & 0.052 & [-0.042, 0.162] \\
730 & kurtosis & 1.258 & 0.018 & [1.229, 1.301] \\
\hline
1825 & mean     & 763.331 & 21.838 & [720.507, 806.251] \\
1825 & s.d.     & 751.068 & 7.732 & [734.755, 765.257] \\
1825 & skewness & 0.558 & 0.058 & [0.446, 0.674] \\
1825 & kurtosis & 1.490 & 0.073 & [1.365, 1.651] \\
\hline
\end{tabular}
\end{table}

As a complementary result, we also present a DTIHT summary, since the finite-dimensional intensity sequence is one of the inferential targets developed in Subsections~\ref{subsec:DTIHT_definition} and~\ref{subsec:DTIHT_curve_joint}. For the discrete-time intensity of the hitting time, the absorbing nature of $\mathrm{RD}$ gives
\begin{equation*}
r(m;M)=\mathbb{P}(Z_{m-1}\in\{\mathrm{Tx},\mathrm{PR}\},Z_m=\mathrm{RD})=\mathbb{P}(T_{\mathrm D}=m),
\qquad 1\leq m\leq M.
\end{equation*}
Daily estimates are sparse, so Table~\ref{tab:ebmt_dtiht_bins} reports the finite-dimensional linear summaries
\begin{equation*}
R(a,b;M^\star)=\sum_{m=a+1}^{b} r(m;M^\star),
\end{equation*}
for selected day intervals. Their covariance can be obtained by summing the corresponding entries of the joint covariance matrix in Theorem~\ref{thm:CLT_r_curve_restricted}. The last row is not a DTIHT ordinate or a linear summary of the DTIHT sequence; it is the complementary calendar-horizon survival mass $\mathbb{P}(T_{\mathrm D}>M^\star)$. The reported intervals are trajectory-bootstrap percentile intervals.

\begin{table}[htbp]
\centering
\caption{Aggregated EBMT discrete-time hitting intensities over selected day intervals and the complementary survival mass, on the retained $M^\star=1825$ day sample.}
\label{tab:ebmt_dtiht_bins}
\begin{tabular}{l r r r r}
\hline
Interval/status & count & estimate & boot. s.e. & 95\% boot. CI \\
\hline
$(0,100]$      & 235 & 0.2009 & 0.0117 & [0.1786, 0.2239] \\
$(100,180]$    & 171 & 0.1462 & 0.0103 & [0.1265, 0.1658] \\
$(180,365]$    & 198 & 0.1692 & 0.0110 & [0.1479, 0.1906] \\
$(365,730]$    & 119 & 0.1017 & 0.0088 & [0.0846, 0.1197] \\
$(730,1095]$   & 51  & 0.0436 & 0.0060 & [0.0325, 0.0556] \\
$(1095,1460]$  & 31  & 0.0265 & 0.0047 & [0.0179, 0.0359] \\
$(1460,1825]$  & 19  & 0.0162 & 0.0037 & [0.0094, 0.0239] \\
No RD by 1825 & 346 & 0.2957 & 0.0133 & [0.2692, 0.3214] \\
\hline
\end{tabular}
\end{table}

The EBMT example uses independent trajectories with an observed intermediate state and an absorbing failed state. It reports calendar-truncated failure-time summaries, aggregated DTIHT probabilities and the complementary calendar-horizon survival mass for the selected complete-case population. Right censoring remains the principal limitation of this illustration; censoring-weighted semi-Markov reliability inference is outside the present scope.

The reporting order is therefore explicit: state the finite-horizon target, report its sampling uncertainty, and then give the corresponding calendar-horizon or backward-recurrence diagnostics according to the target under consideration.


\section{Conclusion}
\label{S:conclusion}

This paper has developed a finite-horizon nonparametric inference framework for reliability indicators of discrete-time semi-Markov systems observed through multiple independent trajectories. The central point is that fixed-horizon data identify a finite restriction of the coupled Markov chain, and therefore identify restricted reliability indicators rather than unrestricted tail-dependent quantities. For these finite-horizon targets, the empirical restricted initial law and the empirical restricted transition matrix have a joint Gaussian limit, and plug-in estimators inherit Gaussian limits through finite-dimensional matrix differentiation.

The restricted reliability indicators considered here are finite-dimensional polynomial and resolvent functionals of the restricted transition matrix. Their differentials give standard errors for restricted moment characteristics, simultaneous confidence envelopes for finite-dimensional discrete-time intensity sequences, Wald-type inference for linear summaries and joint inference for moment-based shape characteristics. The constructive non-identifiability result explains why the unrestricted MTTF is not a nonparametric target under fixed-horizon observation alone, while the finite-order refinements, nested-horizon decomposition, boundary diagnostics and horizon-sensitivity calculations help distinguish sampling uncertainty from finite-horizon effects.

Restricted reliability indicators should be reported as finite-horizon characteristics whenever the diagnostics indicate possible interaction with the unobserved tail. Full unbounded-horizon reliability inference requires additional assumptions on the continuation of the sojourn-time mechanism or on the infinite-state operator. The EBMT example illustrates the distinction between age-restricted and calendar-truncated quantities in an empirical multiple-trajectory setting.

The present analysis is formulated in discrete time. In continuous time, augmentation by the backward recurrence time yields a Markov process on $E\times\mathbb R_+$, but the transition mechanism is described through sojourn-time distributions or hazards rather than finite multinomial rows. Fixed-horizon inference would therefore require a continuous-time likelihood and a martingale or operator-based covariance analysis; these arguments are not developed here. Other reliability indicators, including mean time to repair, mean up time, mean down time, mean time between failures and failure rates, would first require definitions of their finite-horizon targets and the corresponding differentials; see \citet{LimniosOprisan2001} and \citet{Barb08}.

In a hidden semi-Markov model, the physical state is latent and the backward recurrence time, being determined by latent jump times, is not directly observed. The coupled process can still be used as a latent Markov state, but the observed likelihood must be evaluated by filtering or forward--backward recursions. Finite-horizon identifiability then depends on both the latent transition mechanism and the emission distributions, up to the usual label indeterminacy. Consequently, the observed-state row-count estimator and its covariance formula do not carry over directly. Reliability indicators for hidden semi-Markov chains and hidden Markov renewal chains have been studied in relation to the DTIHT \citep{Vots15}; related latent-state formulations are discussed in \citet{PeyrardDeSaporta2026}, with flexible dwell-time models in \citet{PohleAdamBeumer2022} and \citet{RicciottiEtAl2025}. Continuous-time hidden Markov processes would additionally require the filtering methods treated by \citet{Kutoyants2025}.

Resampling and smoothing methods provide a separate inferential extension. The bootstrap, kernel and bootstrapped kernel estimators studied by \citet{VBouz25} could be adapted to restricted reliability indicators, but their analysis would have to preserve the finite-horizon target and account for simultaneous confidence envelopes and horizon diagnostics. Such results are not established in the present paper.

A distinct asymptotic extension concerns an increasing observation horizon. The case with fixed $L$ and $M\to\infty$ is not a direct limit of the present theorem: within-trajectory dependence accumulates, the number of coupled age states grows, and recurrence, mixing and sojourn-tail conditions become essential. A joint regime $M=M_L\to\infty$ and $L\to\infty$ would form a growing-dimensional triangular array. Let ${\cal R}_{M_L}$ denote the rows required by the target at horizon $M_L$. At a minimum, such a regime would require effective exposure conditions such as
\begin{equation*}
\min_{x\in{\cal R}_{M_L}}L\pi_{x\mid M_L}\longrightarrow\infty
\end{equation*}
for the rows required by the target, together with uniform control of the number of such rows, the covariance operators, the derivatives of the reliability functionals and, for resolvent targets, the distance of the relevant spectral radii from one. If the aim is an unbounded-horizon functional, the restriction bias must also vanish at a rate compatible with the sampling error. The fixed-$M$ multiple-trajectory limit, the one-long-trajectory limit and a joint limit therefore require separate asymptotic analyses. A unified theorem is not derived from the present finite-dimensional argument.

\section*{Statements and declarations}

\noindent\textbf{Funding.} The authors acknowledge the support of the French National Research Agency (ANR), under grant ANR-21-CE40--005 (project HSMM-INCA).\par
\vspace{0.5\baselineskip}

\noindent\textbf{Conflict of interest.} The authors declare that they have no conflict of interest.\par
\vspace{0.5\baselineskip}

\noindent\textbf{Data availability.} No new empirical data were collected for this study. The simulated numerical illustration and coverage assessments are generated from the model specified in Section~\ref{S:application}. The EBMT post-transplant data used in the real-data example are publicly available as the \texttt{ebmt3} data set in the \texttt{mstate} package; see \citet{Putter2007} and \citet{deWreedeFioccoPutter2011}.\par
\vspace{0.5\baselineskip}

\noindent\textbf{Code availability.} The code used to generate the numerical tables and figures in Section~\ref{S:application} is available from the corresponding author upon request.\par


\bibliographystyle{abbrvnat}
\bibliography{SMM_MCAP_references}

\end{document}